\documentclass[11pt]{article}
 \usepackage{graphicx,amsmath,amsfonts,amssymb,color,mathrsfs}
\usepackage{leftidx}

\makeatletter
\newcommand*\bigcdot{\mathpalette\bigcdot@{1.4}} % 1.4 = size factor; tweak to taste
\newcommand*\bigcdot@[2]{%
  \mathbin{\vcenter{\hbox{\scalebox{#2}{$\m@th#1\cdot$}}}}%
}
\makeatother

\usepackage[ruled,vlined,linesnumbered]{algorithm2e}
\usepackage{mathtools}
\usepackage{enumitem}

\usepackage{amsmath}

\usepackage{epsfig}

 \newcommand{\QED}{\hfill \thicklines \framebox(6.6,6.6)[l]{}}
 \newenvironment{proof}{\noindent {\bf \sc Proof.} \rm}{\QED}

\numberwithin{equation}{section}

 \newtheorem{theorem}{Theorem}[section]
 \newtheorem{lemma}{Lemma}[section]
 \newtheorem{proposition}{Proposition}[section]
 
 \newtheorem{remark}{Remark}[section]

 \newtheorem{corollary}{Corollary}[section]
 \newtheorem{definition}{Definition}[section]

\newcommand{\eqnb}{\begin{eqnarray*}}
\newcommand{\eqne}{\end{eqnarray*}}

\def\beqlb{\begin{eqnarray}}\def\eeqlb{\end{eqnarray}}
 \def\beqnn{\begin{eqnarray*}}\def\eeqnn{\end{eqnarray*}}

\makeatletter
 \newif\if@borderstar
 \def\bordermatrix{\@ifnextchar*{%
 \@borderstartrue\@bordermatrix@i}{\@borderstarfalse\@bordermatrix@i*}%
 }
 \def\@bordermatrix@i*{\@ifnextchar[{\@bordermatrix@ii}{\@bordermatrix@ii[()]}}
 \def\@bordermatrix@ii[#1]#2{%
 \begingroup
 \m@th\@tempdima8.75\p@\setbox\z@\vbox{%
 \def\cr{\crcr\noalign{\kern 2\p@\global\let\cr\endline }}%
 \ialign {$##$\hfil\kern 2\p@\kern\@tempdima & \thinspace %
 \hfil $##$\hfil && \quad\hfil $##$\hfil\crcr\omit\strut %
 \hfil\crcr\noalign{\kern -\baselineskip}#2\crcr\omit %
 \strut\cr}}%
 \setbox\tw@\vbox{\unvcopy\z@\global\setbox\@ne\lastbox}%
 \setbox\tw@\hbox{\unhbox\@ne\unskip\global\setbox\@ne\lastbox}%
 \setbox\tw@\hbox{%
 $\kern\wd\@ne\kern -\@tempdima\left\@firstoftwo#1%
 \if@borderstar\kern2pt\else\kern -\wd\@ne\fi%
 \global\setbox\@ne\vbox{\box\@ne\if@borderstar\else\kern 2\p@\fi}%
 \vcenter{\if@borderstar\else\kern -\ht\@ne\fi%
 \unvbox\z@\kern-\if@borderstar2\fi\baselineskip}%
 \if@borderstar\kern-2\@tempdima\kern2\p@\else\,\fi\right\@secondoftwo#1 $%
 }\null \;\vbox{\kern\ht\@ne\box\tw@}%
 \endgroup
 }
 \makeatother
\title{\Large \bf GTH Algorithm, Censored Markov Chains, and $RG$-Factorization in Block-Form \\
(dedicated to Dr. Winfried Grassmann)}
\author{Qihui Bu\footnote{School of Communications and Information Engineering, Nanjing University of Posts and Telecommunications, buqihui@njupt.edu.cn} \and Yiqiang Q. Zhao\footnote{School of Mathematics and Statistics, Carleton University, email:zhao@math.carleton.ca}
}
\date{Revised February 2026}

\begin{document}
 \maketitle

\begin{abstract}

In 1985, Grassmann, Taksar, and Heyman~\cite{GTH:1985} published their celebrated paper,
in which they introduced a numerically stable algorithm
for computing the stationary probabilities of a finite-state Markov chain,
one of the key performance quantities in both theory and applications.
This algorithm later became the well-known GTH algorithm (or the state-reduction method)
in the literature, becoming one of the standard algorithms in applied probability.
In 1990, motivated by queueing applications, this algorithm was extended
to deal with the stationary distributions
of block-structured Markov chains with repeating rows by
Grassmann and Heyman~\cite{Grassmann-Heyman:1993}.

In this paper, we focus on the block-form GTH algorithm and organize it into two parts.
In the first part, we connect the block-form GTH algorithm to censored Markov chains
and the block-form $RG$-factorization.
We show that the forward block-elimination and back block-form substitution
of the block-form GTH algorithm are equivalent to solving a
system formulated using the $RG$-factorization in two steps.
We also show that this connection
remains valid
when the block-form GTH algorithm is extended to infinite-state Markov chains.
It is well known that censoring an infinite-state Markov chain to a finite state space
yields a stationary distribution that
provides a best approximation to the stationary distribution of the original
infinite-state Markov chain.

In the second part, we first derive an explicit expression for the censored Markov chain
from the infinite state space to a finite space for Markov chains of $M/G/1$ type.
Based on this expression,
we propose a renormalized approximated censored transition matrix (RA-CM).
The resulting stationary distribution is shown to be asymptotically optimal
in terms of approximation error.
We compare the approximation error of the RA-CM with the error arising from the
last-block-column augmentation.

\vspace*{5mm}


\medskip

\end{abstract}

\section{Introduction}

Stationary probabilities are crucial for system performance,
when the underlying system is modeled using a Markov chain.
Explicit expressions for the stationary distribution are available for only a small set of problems.
In most of the applications, the main tools for performance evaluation and optimization are numerical computations, approximations, or simulations.
Therefore, computational methods for Markov chains are very important.
Algorithms, like the GTH, are usually presented for a finite-state Markov chain,
since computers can only deal with finitely many states.
If the model consists of infinitely many states,
a standard method is to use truncation to approximate the infinite-state transition matrix by a finite-state matrix, and hope the stationary distribution of the finite-state chain is an approximation of the infinite-state chain, for example, the well-known augmented Markov chain approach.

The GTH algorithm discussed in this paper was proposed by Grassmann, Taksar, and Heyman in 1985~\cite{GTH:1985}.
We should also give credit to Sheskin~\cite{Sheskin:1985}, who provided a partitioning algorithm in the same year,
which is equivalent to the GTH algorithm.
Due to the computational stability of the GTH algorithm,
it soon became very attractive to many researchers.
A list of publications by 1993, in which the GTH algorithm was used or studied,
can be found in Grassmann~\cite{Grassmann:1993}, including:
Kohlas~\cite{Kohlas:1986},
Heyman~\cite{Heyman:1987},
Heyman and Reeves~\cite{Heyman-Reeves:1989},
Grassmann and Heyman~\cite{Grassmann-Heyman:1990},
Stewart~\cite{Stewart:1993},
O'Cinneide~\cite{O'Cinneide:1993}.
The list of references closely related to the GTH algorithm,
published after 1993, is very large,
and only a small sample is given here:
Stewart~\cite{Stewart:1994},
Dayar and Stewart~\cite{Dayar-Stewart:1996},
Grassmann and Zhao~\cite{Grassmann-Zhao:1997},
Sonin and Thornton~\cite{Sonin-Thornton:2001},
Dayar and Akar~\cite{Dayar-Akar:2005},
Hunter~\cite{Hunter:2018}.
Further related references can be found in the cited works.

This paper consists of two parts.
The first part is mainly a survey of the block-form GTH algorithm.
In this part, we detail the connections between the block-form GTH algorithm and censored Markov chains.
We explain that the forward block-elimination in the block-form GTH algorithm
is the process of constructing a sequence of censored processes with the censoring set reduced to only one level.
In terms of censoring, the GTH algorithm can be formally extended to infinite-state Markov chains,
but it cannot be practically implemented,
since for an infinite matrix one cannot start an initial elimination step,
which is equivalent to computing a censored matrix with finite state.
It is a literature result (Zhao and Liu~\cite{Zha:96}) that
when augmentations are used for approximating the stationary distribution of an infinite-state Markov chain,
the censored chain is best in approximation error under the $\ell_1$ norm.
In this part, we also show that the forward elimination and back substitution steps of the GTH algorithm
amount to the same process of solving the equations $I-P=0$,
where $P$ is the transition matrix,
in two steps by using the $RG$-factorization of $I-P$. This equivalence was observed in Zhao~\cite{Zhao:2020} for the scalar case without detailed discussion.
The connection between the GTH algorithm and censored Markov chains
existed since the introduction of the GTH algorithm (see \cite{Sheskin:1985}).
Later, in \cite{Grassmann-Heyman:1990}, this connection was discussed explicitly for block-structured Markov chains,
and in the same paper,
the censored Markov chain was used to formally extend the GTH algorithm to infinite-state Markov chains.

When the stationary distribution of an infinite-state Markov chain is our focus,
based on the discussion in the first part,
it is clear that the censored chain,
or the finite transition matrix constructed by the formal initial elimination,
gives a distribution with minimal approximation error.
However, in general, this initial elimination is not computationally available.
The focus of the second part of this paper is to compute a finite transition matrix for Markov chains of $M/G/1$ type
(referred to as the renormalized approximated censored matrix, or simply RA-CM),
which is an approximation to the censored chain.
Therefore, it is asymptotically best in approximation error.
Both the censored Markov chain and our proposed approximation
are special cases of augmentations,
which are transition matrices obtained by augmenting the north-west corner
of the original transition matrix to make it stochastic.
Studies of augmentations date back to the 1960s,
for example, the work of Seneta~\cite{Seneta:1967, Seneta:1968},
which focused on conditions under which the stationary distribution
of the augmented chain approximates that of the original infinite-state chain.
Among all augmentations,
the last-column (or last-block-column) augmentation approach deserves special attention,
since it is very easy to implement and often leads to approximations with small errors.
For example, it was proved that when the block-structured transition matrix is block monotone,
the natural-last-block-column augmentation also produces a best approximation
(see Li and Zhao~\cite{Li-Zhao:2000}).
In the second part,
we also compare approximation errors of our proposed RA-CM
and the natural-last-block-column augmentation
for a queueing system with batch arrivals.

In the following, we highlight our main new contributions in this paper:
\begin{description}
  \item[1.] The first part (Sections~2--4) is mainly a survey of the GTH algorithm,
  together with basics of censored Markov chains and the $RG$-factorization.
  Our new contribution is:
  \begin{description}
    \item[(a)] The equivalence between the forward block-eliminations
    and the back block-form substitutions in the block-form GTH algorithm
    and a two-step solution provided by the $RG$-factorization
    (see Section~\ref{sec:RG-F}).
  \end{description}

  \item[2.] The main contributions in the second part (Sections~5--7) include:
  \begin{description}
    \item[(b)] Construction of an approximation (referred to as RA-CM)
    to the censored process to any desired accuracy
    (see Theorem~\ref{thm:truncation-error}).
    The stationary distribution of the RA-CM is asymptotically optimal under the $\ell_1$ norm
    as an approximation to the stationary distribution
    of the original Markov chain of $M/G/1$ type
    among all augmentation methods
    (see Section~\ref{sec:RA-CM}).
    \item[(c)] An explicit expression for the censored matrix
    of Markov chains of $M/G/1$ type
    (see Theorem~\ref{thm:QmD-censored-clean}).
    This result plays a key role in the construction of the RA-CM
    and is also of independent interest.
    \item[(d)] Proof of positive recurrence of the RA-CM
    (see Theorem~\ref{the:convergence-rowwise}).
    \item[(e)] Theorem~\ref{thm:global-l1},
    which guarantees convergence of the stationary distributions
    of the natural-last-block-column augmentation
    to that of the original $M/G/1$ type Markov chain.
    This result is needed for approximation-error comparisons
    in the application example of Section~\ref{sec:example}.
    \item[(f)] An application presented in Section~\ref{sec:example},
    namely an $M^X/M/1$ queueing system with batch arrivals.
    A special case allowing only one service rate during working vacations
    was studied in \cite{Baba:2012}.
  \end{description}
\end{description}

The rest of the paper is organized as follows:
the first part of this paper consists of Sections~\ref{sec:2}--\ref{sec:RG-F}.
The block-form GTH algorithm is introduced and discussed in Section~\ref{sec:2};
the censored Markov chain and its main properties are reviewed in Section~\ref{sec:censoring}.
In terms of censoring,
the GTH algorithm is formally extended to infinite-state Markov chains,
and the forward elimination process is equivalent to constructing a sequence of censored chains.
The block-form $RG$-factorization is generalized to a block-partitioned transition matrix
in Section~\ref{sec:RG-F},
which is a probabilistic counterpart of the $LU$-decomposition
in Gaussian elimination.
We show the connection between the $RG$-factorization and the GTH algorithm.

The second part of this paper consists of Sections~\ref{sec:RA-CM}--\ref{sec:example}.
A renormalized approximated (truncated) censored matrix (RA-CM)
is proposed in Section~\ref{sec:RA-CM}, which is asymptotically best in approximation;
last-block-column augmentations (LBCA)
are introduced in Section~\ref{sec:LBCA}
for comparison with the RA-CM;
a queueing system with batch arrivals and vacations
is studied in Section~\ref{sec:example},
which generalizes earlier results in the literature.
We compare the approximation performance of the RA-CM and the natural-LBCA.

The final section, Section~\ref{sec:final},
contains concluding remarks.

\section{Block-form GTH algorithm} \label{sec:2}

The (standard) GTH algorithm, consisting of two steps: forwarding state reduction (or variable elimination) and backward substitution, is a numerically stable algorithm for computing the stationary distribution of a finite-state irreducible Markov chain. Each step of the GTH algorithm has a probabilistic interpretation.
Mathematically, the GTH algorithm is a rearrangement of Gaussian elimination. The GTH algorithm is numerically stable, since it starts with the smallest components and avoids subtractions, while the (standard) Gaussian elimination can become numerically unstable if the number of states becomes large. For the standard GTH algorithm, readers can refer to \cite{GTH:1985}. Our purpose is to connect the block-form GTH algorithm to the censored Markov chain in block-form, and the block-form $RG$-factorization. With the connection to the censored chain, the block-form GTH algorithm can be formally extended to infinite-state Markov chains.

We start by introducing Markov chains in block-form. Let $S$ be the state space defined as
\[
    S= \cup_{i=0}^\infty L_i,
\]
where $L_i = \{ (i,\alpha) : \alpha = 1,2,\ldots, r\}$. Here $i$ and $\alpha$ are referred to as the level and phase variables, respectively.
Let $P$ be the transition probability matrix of a Markov chain, partitioned according to the level and referred to as a Markov chain (or transition matrix) in block-form:
\begin{equation} \label{eqn:Pij}
    P =\left [
\begin{array}{cccc}
P_{0,0} & P_{0,1} & P_{0,2} & \cdots \\
P_{1,0} & P_{1,1} & P_{1,2} & \cdots \\
P_{2,0} & P_{2,1} & P_{2,2} & \cdots \\
\vdots & \vdots & \vdots & \ddots%
\end{array}
\right ].
\end{equation}

The block-form GTH algorithm is an extension of the (standard) GTH algorithm for computing the unique stationary distribution
\[
    \pi = (\pi_0, \pi_1, \ldots, \pi_N)
\]
with $\pi_i =(\pi_{i,1}, \pi_{i,2}, \ldots, \pi_{i,r})$ ($i=0, 1, \ldots, N$)
of an irreducible block-form Markov chain $P$ with finite levels:
\begin{equation} \label{eqn:Pij-N}
    P = \left [  \begin{array}{ccccc}
P_{0,0} & P_{0,1} & P_{0,2} & \cdots \cdots & P_{0,N} \\
P_{1,0} & P_{1,1} & P_{1,2} & \cdots \cdots & P_{1,N} \\
P_{2,0} & P_{2,1} & P_{2,2} & \cdots \cdots & P_{2,N} \\
\vdots  & \vdots  & \vdots  &  \vdots       & \vdots \\
P_{N,0} & P_{N,1} &P_{N,2} & \cdots \cdots & P_{N,N} \end{array}
       \right ].
\end{equation}
Equivalently, it computes nonnegative $\pi$ such that
\[
    \pi = \pi P \quad \text{with} \quad \sum_{i=0}^N \sum_{\alpha=1}^{r} \pi_{i,\alpha}=1.
\]
Writing out in detail, the stationary equations $\pi = \pi P$ are
\begin{equation} \label{eqn:se}
    \pi_j = \sum_{i=0}^N \pi_i P_{i,j}, \quad j =0, 1, \ldots, N.
\end{equation}
Since all probabilities $\pi_{i,\alpha}$ sum up to one, one of the above $(N+1)r$ equations is redundant.

\noindent \textbf{Block-form GTH algorithm:} The algorithm consists of two parts, forward block-elimination and back block-form substitution.

\textbf{Forward block-eliminations:}  First, starting from the last level, level $n=N$, use the equation in block-form given in (\ref{eqn:se}) to eliminate $\pi_N$ to obtain a system of equations in block-form for $\pi_0$, $\pi_1$, \ldots, $\pi_{N-1}$ corresponding to $N$ levels from 0 to $N-1$. The resulting system from the first block-elimination is given as
\[
    \pi_j = \sum_{i=0}^{N-1} \pi_i P_{i,j}^{\langle N-1 \rangle}, \quad j =0, 1, \ldots, N-1,
\]
where the new coefficients $P_{i,j}^{\langle N-1 \rangle}$ in block-form are determined by
\begin{equation} \label{eqn:forward-(N-1)}
    P_{i,j}^{\langle N-1 \rangle} = P_{i,j}^{\langle N\rangle} + P_{i,N}^{\langle N \rangle} (I-P_{N,N}^{\langle N\rangle})^{-1}P_{N,j}^{\langle N\rangle}
\end{equation}
with $P_{i,j}^{\langle N\rangle} = P_{i,j}$.

It can be directly verified that
\begin{equation} \label{eqn:N-1}
    P^{\langle N-1\rangle} :=  \left [  \begin{array}{cccc}
P_{0,0}^{\langle N-1\rangle} & P_{0,1}^{\langle N-1\rangle} & \ldots \ldots & P_{0,N-1}^{\langle N-1\rangle} \\
P_{1,0}^{\langle N-1\rangle} & P_{1,1}^{\langle N-1\rangle} & \ldots \ldots &P_{1,N-1}^{\langle N-1\rangle} \\
   \vdots & \vdots & \vdots& \vdots \\
P_{N-1,0}^{\langle N-1\rangle} & P_{N-1,1}^{\langle N-1\rangle} & \ldots \ldots & P_{N-1,N-1}^{\langle N-1\rangle} \end{array}
       \right ]
\end{equation}
is a stochastic matrix, and hence defines a new Markov chain.

Continue the block-form elimination process to obtain $\pi_{N-1}$, $\pi_{N-2}$, \ldots, $\pi_{n}$ to have the following system of equations in block-form:
\begin{equation} \label{eqn:se-j}
    \pi_j = \sum_{i=0}^{n-1} \pi_i P_{i,j}^{\langle n-1\rangle}, \quad j =0, 1, \ldots, n-1,
\end{equation}
where the coefficients $ P_{i,j}^{\langle n-1\rangle}$ in block-form are determined by
\begin{equation}\label{eqn:iterations-old}
      P_{i,j}^{\langle n-1\rangle} = P_{i,j}^{\langle n\rangle} + P_{i,n}^{\langle n\rangle} (I-P_{n,n}^{\langle n\rangle})^{-1}P_{n,j}^{\langle n\rangle}.
\end{equation}
The matrix of these coefficients:
\begin{equation} \label{ean:n-1}
    P^{\langle n-1\rangle} :=  \left [  \begin{array}{cccc}
P_{0,0}^{\langle n-1\rangle} & P_{0,1}^{\langle n-1\rangle} & \ldots \ldots & P_{0,n-1}^{\langle n-1\rangle} \\
P_{1.0}^{\langle n-1\rangle} & P_{1,1}^{\langle n-1\rangle} & \ldots \ldots & P_{1,n-1}^{\langle n-1\rangle} \\
   \vdots & \vdots & \vdots& \vdots \\
P_{n-1,0}^{\langle n-1\rangle} & P_{n-1,1}^{\langle n-1\rangle} & \ldots \ldots & P_{n-1,n-1}^{\langle n-1\rangle} \end{array}
       \right ]
\end{equation}
defines a new Markov chain.

After $\pi_1$ is eliminated, we reach a Markov chain corresponding to level 0, given by
\[
   \pi_0 = \pi_0  P_{0,0}^{\langle 0\rangle}.
\]

\textbf{Back block-form substitutions:}
To compute $\pi_j$, the block-form GTH algorithm performs the back block-form substitution. % The first step is different from the standard GTH algorithm, since in the scalar case $\pi_0 = \pi_0  P_{0,0}^{\langle 0\rangle} = \pi_0 \cdot 1$, which is redundant, while in block case, only one of the $r$ equations in $\pi_0 = \pi_0  P_{0,0}^{\langle 0\rangle}$ is redundant. One option is to perform the standard GTH algorithm on the equations in $\pi_0 = \pi_0  P_{0,0}^{\langle 0\rangle}$.
In the scalar case, $\pi_0 = \pi_0  P_{0,0}^{\langle 0\rangle} = \pi_0 \cdot 1$ is redundant in back substitution, while only one of the $r$ equations in $\pi_0 = \pi_0  P_{0,0}^{\langle 0\rangle}$ is redundant in the block case. Thus, the implementation of back substitution in the block case differs from that in the scalar case, particularly in the first step.
One option is to perform the standard GTH algorithm on the equations in $\pi_0 = \pi_0  P_{0,0}^{\langle 0\rangle}$. For this purpose, we denote
\[
    P_{0,0}^{\langle 0\rangle} =
\left [  \begin{array}{cccc}
p_{1,1}^{\langle 0\rangle} & p_{1,2}^{\langle 0\rangle} & \cdots \cdots & p_{1,r}^{\langle 0\rangle} \\
p_{2,1}^{\langle 0\rangle} & p_{2,2}^{\langle 0\rangle} & \cdots \cdots & p_{2,r}^{\langle 0\rangle} \\
   \vdots & \vdots & \vdots  & \vdots \\
p_{r,1}^{\langle 0\rangle} & p_{r,2}^{\langle 0\rangle} & \cdots \cdots & p_{r,r}^{\langle 0\rangle} \end{array}
       \right ].
\]
The forward (scalar) elimination, after the elimination of $\pi_{0,r}$, $\pi_{0,r-1}$, \ldots, $\pi_{0,\gamma}$, gives
\begin{equation} \label{eqn:scalar-elimination}
    \pi_{0,\beta} = \sum_{\alpha=1}^{\gamma-1} \pi_{0,\alpha} p_{\alpha,\beta}^{\langle 0, \gamma-1\rangle}, \quad \beta = 1, 2, \ldots, \gamma-1,
\end{equation}
where $p_{\alpha,\beta}^{\langle 0, \gamma-1\rangle}$ is determined by
\[
    p_{\alpha,\beta}^{\langle 0,\gamma-1\rangle} = p_{\alpha,\beta}^{\langle 0, \gamma \rangle} + \frac{p_{\alpha,\gamma }^{\langle 0, \gamma \rangle} p_{\gamma,\beta}^{\langle 0, \gamma \rangle}}{\sum_{k=1}^{\gamma-1} p_{\gamma,k}^{\langle 0, \gamma \rangle}}, \quad \alpha, \beta = 1, 2 \ldots, \gamma -1,
\]
with $p_{\alpha,\beta}^{\langle 0,r \rangle}=p_{\alpha,\beta}^{\langle 0 \rangle}$. When $\gamma$ reaches 2, the equation becomes
\[
    \pi_{0,1} = \pi_{0,1} p_{1,1}^{\langle 0,1 \rangle} = \pi_{0,1} \cdot 1 = \pi_{0,1},
\]
which is redundant. Therefore, we can define
\[
    r_0 = (r_{0,1}, r_{0,2}, r_{0,3}, \ldots, r_{0,r}),
\]
where $r_{0,1}=1$ and $r_{0,\alpha}= \pi_{0,\alpha}/\pi_{0,1}$ for $\alpha =2, 3, \ldots, r$. It is easy to see that $r_{0,\beta}$ also satisfies the equation in \eqref{eqn:scalar-elimination}, or for $\gamma = r, r-1, \ldots, 1$,
\[
    r_{0,\beta} = \sum_{\alpha=1}^{\gamma} r_{0,\alpha} p_{\alpha,\beta}^{\langle 0, \gamma\rangle}, \quad \beta = 1, 2, \ldots, \gamma.
\]

Now, with $r_{0,1}=1$, perform the (scalar) back substitution to get the following recursive determination of $r_{0,2}$:
\[
    r_{0,2} = r_{0,1} \;  \frac{p_{1,2}^{\langle 0,2\rangle}}{p_{2,1}^{\langle 0,2\rangle}}.
\]
Continue the back substitution until all $r_{0,\beta}$ are determined by
\[
    r_{0,\beta} = \sum_{\alpha=1}^{\beta-1} r_{0,\alpha} \frac{p_{\alpha,\beta}^{\langle 0,\beta\rangle}}{\sum_{k=1}^{\beta-1} p_{\beta,k}^{\langle 0,\beta\rangle}}, \quad \beta = 2, 3, \ldots, r.
\]

Next, for $i= 1, 2, \ldots, N$, define $r_i=\frac{\pi_i}{\pi_{0,1}}$.  Since $r_{i}$ and $\pi_i$ are different only by a constant, $r_i$ also satisfies equations \eqref{eqn:se} and \eqref{eqn:se-j}. With $r_0$ computed above, the back block-form substitution is now performed to compute $r_j$, for $j=1, 2, \ldots, N$, recursively, as
%\begin{equation*}%\label{}
%  r_2=r_1P_{1,2}^2(I-P_{2,2}^2)^{-1};
%\end{equation*}
\begin{equation}
r_j=\sum\limits_{i=0}\limits^{j-1}r_iP_{i,j}^{\langle j\rangle}(I-P_{j,j}^{\langle j\rangle})^{-1}.
\end{equation}
%(\begin{equation}%\label{}
%   r_1P_{1,j}^{j}+r_2P_{2,j}^j+\cdots+r_jP_{j,j}^j=r_j.
% \end{equation})
Finally,
\begin{equation} \label{eqn:rj}
    \pi_j = \frac{r_j}{\sum_{i=0}^N  \sum_{\alpha=1}^{r} r_{i,\alpha}}, \quad j = 0, 1, \ldots, N.
\end{equation}
%Recall that $r1=(r_{11}, r_{12}, \dots, r_{1N})$ and $\pi_1=(\pi_{11}, \pi_{12}, \ldots, \pi_{1N})$ are different only by a constant, and $\pi_1$ is a probability vector, we can easily normalize $r$ to have
%\begin{equation} \label{eqn:rj}
%    \pi_j = \frac{r_j}{\sum_{k=1}^N r_k}, \quad j = 1, 2, \ldots, N.
%\end{equation}

\begin{remark} It is worthwhile to mention that the forward elimination starts from $\pi_N$, not from $\pi_0$. This is important for the control of computational errors, since for large $n$, tail probabilities are smaller. Also, in the scalar case, $I-P_{n,n}^{\langle n\rangle}$ is replaced by $\sum_{k=1}^{n-1} P_{n,k}^{\langle n\rangle}$ to avoid the subtraction. However, we cannot do so for the block case. Instead, various numerically stable algorithms can be used for computing $(I-P_{n,n}^{\langle n\rangle})^{-1}$. One option is to apply the (standard) GTH algorithm to the (scalar) equations within the block recursively to get the inverse.
\end{remark}

\begin{remark}
It is well-known that the standard GTH algorithm is mathematically equivalent to Gaussian elimination. As indicated earlier, Gaussian elimination can be numerically unstable when $N$ is large, while the GTH algorithm is very stable.
This property remains valid for the block-form GTH algorithm.
\end{remark}

\begin{remark}
The block-form GTH algorithm can also be discussed in an analogous way for a continuous time Markov chain.
\end{remark}

\section{Censored Markov chains} \label{sec:censoring}

In this section, we first detail how the forward block-form elimination in the block-form GTH algorithm is indeed the process of computing a sequence of censored Markov chains with censoring sets $L_{\leq n} := \cup_{i=0}^n L_i$ for $i = N, N-1, \ldots, 0$. When $i=N$, the censored process is simply the original process, or $P$ in \eqref{eqn:Pij-N}. Since the standard GTH algorithm is a special case of the block-form GTH algorithm with block size 1, the forward elimination in the standard GTH algorithm is simply the process of computing a sequence of censored processes with censoring sets $\{0, 1, \ldots, N \}$, $\{0, 1, \ldots, N-1 \}$, \ldots, $\{0\}$. Then, we explain how the GTH algorithm can be formally extended to an infinite-state Markov chain $P$ in \eqref{eqn:Pij}.

\begin{definition}[Censored process]
Let $\{X_n: n =0, 1, 2, \ldots \}$ be a countable-state Markov chain with state space
$S$ and transition probability matrix $P$.
Let $E$ be a non-empty subset of $S$. Suppose that $n_k$ is the $k$-th time at which the original process $X_n$ visits the subset $E$. Then, the censored stochastic process $\{X_k^E: k =0, 1, 2, \ldots \}$ is defined by $X_k^E=X_{n_k}$, i.e., the value of the new process $X_k^E$ at time $k$ is equal to the value of the original process $X_n$ at its {$k$-th} time of visiting the subset $E$.
\end{definition}

The censored process is also referred to as a watched process since it is obtained by watching $X_n$ only when it is in $E$. It is also referred to as an embedded Markov chain since the time (or the state space) of the censored process is embedded in the time (or the state space) of the original process.
The following lemma is a summary of some basic properties of the censored process.
\begin{lemma} \label{lem:censored}
\textbf{(i)} The censored process $X_k^E$ is also a Markov chain. If the transition probability matrix $P$ of the Markov chain $X_n$ is partitioned according to $E$ and its complement $E^c$:
\[
    P = \bordermatrix[{[]}]{%
    & E & E^c \cr
E & T & U \cr
E^c & D & Q
},
\]
then the transition probability matrix of the censored Markov chain $X_k^E$ is given by
\begin{equation} \label{eqn:censor-E}
    P^E = T + U \hat{Q} D,
\end{equation}
where $\hat{Q} = \sum_{n=0}^{\infty} Q^n$ is the minimal nonnegative inverse of $I-Q$.

\textbf{(ii)}
The Markov chain $P$ is irreducible, if and only if for every subset $E$ of the state space, the censored Markov chain $P^E$ is irreducible. If $P$ has a unique stationary probability vector $\pi=(\pi_1, \pi_2, \ldots)$, then the stationary probability vector $\pi^E = \big (\pi_j^E \big )_{j \in E}$ of the censored Markov chain is given by
\[
	\pi^E_j = \frac{\pi_j}{\sum_{k \in E} \pi_k},\quad j \in E.
\]

\textbf{(iii)} If $E_1$ and $E_2$ are two non-empty subsets of the state space $S$ and $E_2$ is a subset of $E_1$, then
\[
    P^{E_2} = (P^{E_1})^{E_2}.
\]
\end{lemma}

For connecting the block-form GTH algorithm with censored Markov chains, we adapt the notation for censored processes to the block-partitioned transition matrix in \eqref{eqn:Pij-N} by considering the censoring set $L_{\leq n}$. First, consider the censored process with censoring set $L_{\leq N-1}$. Partition $P$ in \eqref{eqn:Pij-N} according to $L_{\leq N-1}$ and $L_{\leq N-1}^c=L_N$:
\[
    P = \bordermatrix[{[]}]{%
    & L_{\leq N-1} & L_N \cr
L_{\leq N-1} & T_{N-1} & U_{N-1} \cr
L_N & D_{N-1} & P_{N,N}
}.
\]
Then, the censored matrix with censoring set $L_{\leq N-1}$ is, according to \eqref{eqn:censor-E}, given by
\[
    P^{(N-1)} := P^{L_{\leq N-1}} = T_{N-1} + U_{N-1} (I - P_{N,N})^{-1} D_{N-1}.
\]
Let the entries of $P^{(N-1)}$ be $P^{(N-1)}_{i,j}$, so that $P^{(N-1)}=(P^{(N-1)}_{i,j})_{i,j = 0, 1, \ldots, N-1}$. Write this out in detail, we have
\[
    P^{(N-1)}_{i,j} = P_{i,j} + P_{i,N} (I-P_{N,N})^{-1} P_{N,j}, \quad i, j = 0, 1, \ldots, N-1.
\]
Compare this with the first forward block-elimination in the block-form GTH algorithm, given in \eqref{eqn:forward-(N-1)}, and notice that $P_{i,j}^{\langle N\rangle} = P_{i,j}$, we conclude that
\[
    P^{(N-1)}_{i,j} = P_{i,j}^{\langle N-1 \rangle}, \quad i,j =0, 1, \ldots, N-1,
\]
or the censored matrix $P^{(N-1)}$ is the same as $P^{\langle N-1 \rangle}$, which is the matrix after the first step forward block-form elimination given in \eqref{eqn:N-1}. If, based on $P^{(N-1)}$, we continue the censoring process, and use the property (Lemma~\ref{lem:censored}-(iii)) that
\begin{equation}\label{eqn:iterations-new}
      P^{(n)} := P^{L_{\leq n}} = ((P^{(N-1)})^{(N-2)} \cdots)^{(n)}, \quad n =N-1, N-2, \ldots, 0,
\end{equation}
we can conclude that the forward block-form elimination is the same as computing the sequence of censored matrices for level $n=N-1, N-2, \ldots, 0$:
\[
    P^{(n)} = P^{\langle n \rangle}.
\]

The GTH algorithm applies to finite-state Markov chains. In applications, we often need to compute the stationary distribution of an infinite-state Markov chain, like $P$ in \eqref{eqn:Pij}. The block-form GTH algorithm can be formally extended to the block-partitioned transition matrix $P$ in \eqref{eqn:Pij} in terms of censoring (for example, see \cite{Grassmann-Heyman:1990}). In more detail, partition $P$ in \eqref{eqn:Pij} as
\begin{equation} \label{eqn:Pij-partition}
    P = \bordermatrix[{[]}]{%
    & L_{\leq N} & L_{\leq N}^c \cr
L_{\leq N} & T_{N} & U_{N} \cr
L_{\leq N}^c & D_{N} & Q_{N}
}.
\end{equation}
The formal extension of the block-form GTH algorithm is the following censored Markov chain
\begin{equation}\label{eqn:censored-from-infinite}
      P^{(N)} = T_N + U_N \hat{Q}_N D_N.
\end{equation}

Since the forward elimination step is equivalent to constructing a sequence of censored Markov chains, we may write equation \eqref{eqn:iterations-old} (replacing $n-1$ by $N$) as
\[
    P_{i,j}^{(N)} = P_{i,j}^{(N+1)} + P_{i,N+1}^{(N+1)} (I-P_{N+1,N+1}^{(N+1)})^{-1} P_{N+1,j}^{(N+1)}.
\]
Repeat the iteration process on $P_{i,j}^{(N+1)}$, $P_{i,j}^{(N+2)}$, \ldots to write the block entries $P^{(N)}_{i,j}$ of $P^{(N)}$ as
\begin{equation}\label{eqn:censored-new}
   P^{(N)}_{i,j} = P_{i,j} + \sum_{k=N+1}^{\infty} P^{(k)}_{i,k} (I - P^{(k)}_{k,k})^{-1} P^{(k)}_{k,j}.
\end{equation}

The censored Markov chain given in \eqref{eqn:iterations-new} is called a formal extension of the GTH algorithm, since we cannot perform an initial elimination step. In fact, in general, there is no explicit formula or easy way for computing the above censored matrix (from infinite to finite states).
In practice,
if $P$ is infinite-state, to implement the GTH algorithm various augmentation methods have been proposed to replace the censored chain. This belongs to a separate topic: approximations to the stationary distribution of an infinite-state Markov chain by a finite-state Markov chain, or by a finite-state augmented Markov chain. For the Markov chain in \eqref{eqn:Pij} partitioned as in \eqref{eqn:Pij-partition}, an augmentation is a finite-level Markov chain with its transition matrix
\begin{equation} \label{eqn:Pij-augmented}
    {}_{(N)}\widetilde P = T_N + \widetilde{A},
\end{equation}
where $\widetilde{A}$ is a nonnegative matrix such that ${}_{(N)}\widetilde P$ is stochastic. Suppose that ${}_{(N)}\widetilde P$ has a unique stationary distribution ${}_{(N)}\widetilde{\pi}$. Various easily implementable augmentations exist, including first-column, last-column, and linear augmentations.
When we use an augmentation as step~0 elimination in order to perform the GTH algorithm,
it is  important to ask the following questions:
\begin{description}
    \item[(1)] When does the distribution of ${}_{(N)}\widetilde{\pi}$ have a limit as $N \to \infty$?

        In general, we need to impose conditions for the existence of a limit.

    \item[(2)] If the limit exists, does this limit equal to the stationary distribution of the original Markov chain $P$ (assuming that $P$ is positive recurrent), or can the probability distribution of the augmented chain be an approximation to the stationary distribution of the original chain $P$?

        In general, it cannot be guaranteed that the limit coincides with the stationary distribution of the original chain.

    \item[(3)] When the stationary distribution of the augmentation is an approximation to the stationary distribution of the original chain (or when the answers to (1) and (2) are confirmed), which augmentation is the best in terms of the approximation error for a fixed truncation size $N$?

    The stationary vector of the censored Markov chain in \eqref{eqn:Pij-partition} is given in Lemma~\ref{lem:censored}-(ii), which always has a limit equal to the stationary vector of the original process $P$, and the censored Markov chain is a best augmentation in $\ell_1$ norm (see below for details).
\end{description}

For the scalar case, the censored process gives the minimal error in $\ell_1$ norm compared to all other augmentations, proved in \cite{Zha:96}. This property remains true for the block-partitioned $P$ given in \eqref{eqn:Pij}. To state this property, define the $\ell_1$-norm of the total error between the stationary distribution ${}_{(N)}\widetilde{\pi}$ of an augmented Markov chain ${}_{(N)}\widetilde{P}$ given in \eqref{eqn:Pij-augmented} and the stationary distribution $\pi$ of the original Markov chain $P$ given in \eqref{eqn:Pij}:
\begin{align*}
  \ell_1(N,\infty) &= \sum_{n=0}^{N} || {}_{(N)}\widetilde{\pi} - \pi_n||_1 + \sum_{n=N+1}^{\infty} ||\pi_n||_1 \\
   &=  \sum_{n=1}^{N} \sum_{\alpha=1}^{r}  |{}_{(N)}\widetilde{\pi} - \pi_{n,\alpha}| + \sum_{n=N+1}^{\infty} \sum_{\alpha=1}^{r} \pi_{n,\alpha}.
\end{align*}

\begin{theorem}[Best augmentation, \cite{Zha:96}]
For a fixed level $N$, the censored Markov chain $P^{(N)}$ is an augmentation in block-form such that the error sum $\ell_1(N, \infty)$ is minimized.
\end{theorem}

It is clear that the censored Markov chain is an augmented Markov chain.
In Sections~\ref{sec:RA-CM}, we propose an approximation to the censored matrix for Markov chains of $M/G/1$ type, referred to as the renormalized approximated censored matrix (RA-CM), such that the stationary distribution is asymptotically best (see Section~\ref{sec:RA-CM} for details).

The following are some historical notes on censored Markov chains and augmentations.
The concept of the censored Markov chain was first introduced and studied by L\'{e}vy~\cite{Lev:1951,Lev:1952,Lev:1958}. It was then used by Kemeny, Snell and Knapp~\cite{Kem66}
for proving the uniqueness of the invariant vector for a recurrent countable-state Markov chain. This embedded Markov chain was an approximation tool in the book by Freedman~\cite{Fre:1983} for countable-state Markov chains. Related to the current paper, readers may refer to the following references: Zhao and Liu~\cite{Zha:96}, Zhao \textit{et al.}~\cite{ZLB:1998,ZLB:2003}, Zhao~\cite{Zhao:2000}.
Readers may also refer to  the book by Latouche and Ramaswami~\cite{LR:1987} for censored Markov chains in continuous time.

Studies of the convergence of probability distributions of augmented Markov chains, or error bounds between these two distributions,  date back to the 1960s. Early studies include
Seneta~\cite{Seneta:1967,Seneta:1968,Seneta:1980} and his book \cite{Seneta:1981},
Tweedie~\cite{Tweedie:1971,Tweedie:1998},
Golub and Seneta~\cite{Golub-Seneta:1973,Golub-Seneta:1974},
Wolf~\cite{Wolf:1975,Wolf:1980},
Allen \textit{et al.}~\cite{AAS:1977},
Keilson and Ramaswami~\cite{Keilson-Ram:1984},
Gibson and Seneta~\cite{Gib:1987a,Gib:1987b}; later studies include
Grassmann and Heyman~\cite{Grassmann-Heyman:1993},
Zhao and Liu~\cite{Zha:96},
Li and Zhao~\cite{Li-Zhao:2000},
Cho and Meyer~\cite{ChoMeyer01}; and recent ones include
Liu~\cite{Liu:2010},
Masuyama~\cite{Masuyama:2015,Masuyama:2016,Masuyama:2017a,Masuyama:2017b,Masuyama:2019},
Liu and Li~\cite{ll18},
Liu \textit{et al.}~\cite{LLZ:2021},
Infanger \textit{et al.}~\cite{IGL:2022}.

\section{$RG$-factorization} \label{sec:RG-F}

%\blue{(Update this section to block-form. Also, add details on how Gaussian elimination is equivalent to LU-factorization, then add details on how GTH is equivalent to $RG$-factorization.)}

The standard GTH algorithm computes the finite unknowns of a linear system. In terms of censoring, the GTH algorithm can be formally extended to infinite-state Markov chains. The new contribution in this section is the equivalence between the block-form GTH algorithm (standard and extended) and the $RG$-factorization of $I-P$.

It is well-known that mathematically, the (standard) GTH algorithm is equivalent to Gaussian elimination, which is in turn equivalent to an $LU$-decomposition. Since the forward elimination in the GTH algorithm starts from the last unknown, the decomposition is in the form of $UL$ rather than $LU$. In the following, we first recall two important measures in block-form, dual to each other and referred to as $R$- and $G$-measures in block-form, which are two probability quantities defined for the Markov chain $X_n$ with transition probability matrix $P$ given in \eqref{eqn:Pij}.

\begin{remark}[A reminder on linear algebraic factorizations.]
For a finite matrix $A$, Gaussian elimination without pivoting produces a factorization $A=LU$, where $L$ is unit lower triangular and $U$ is upper triangular.
If the elimination is performed in reverse order (starting from the last variable), then the same algebraic operations yield a factorization of the form $A=UL$, with $U$ unit upper triangular and $L$ lower triangular.
This is the form that naturally matches the forward elimination order in the GTH algorithm.
\end{remark}

For \(0 \leq i < j\), \(R_{i,j}\) is defined as a matrix of size \(r \times r\) whose \((\alpha, \beta)\)th entry is the expected number of visits to state \((j, \beta)\) before hitting any state in \(L_{\leq (j-1)}\), given that the process starts in state \((i, \alpha)\), or
\[
    R_{i,j}(\alpha, \beta) = E[\text{number of visits to } (j, \beta) \text{ before hitting } L_{\leq (j-1)} | X_0 = (i, \alpha)].
\]
The family of matrices $\{ R_{i,j} \}_{j > i \geq 0}$ is referred to as the $R$-measure.

For \(i > j \geq 0\), \(G_{i,j}\) is defined as a matrix of size \(r \times r\) whose \((\alpha, \beta)\)th entry is the probability of hitting state \((j, \beta)\) when the process enters \(L_{\leq (i-1)}\) for the first time, given that the process starts in state \((i, \alpha)\), or
\[
    G_{i,j}(\alpha, \beta) = P[\text{hitting } (j, \beta) \text{ upon entering } L_{\leq (i-1)} \text{ for first time } | X_0 = (i, \alpha)].
\]
The family of matrices $\{ G_{i,j} \}_{0 \leq j < i}$ is referred to as the $G$-measure.

One of the most important properties for the $R$- and $G$-measures is the invariant property under censoring stated in the next theorem. This property is the key for connecting the $UL$-decomposition constructed from the block-form GTH algorithm and the $RG$-factorization.

Readers, who are interested in knowing more about the $R$- and $G$-measures, including the $RG$-factorization, can refer to the following papers and related references wherein:
Grassmann and Heyman~\cite{Grassmann-Heyman:1990},
Heyman~\cite{Heyman:1995},
Zhao, Li and Braun~\cite{ZLB:1997, ZLB:2003},
Zhao~\cite{Zhao:2000}, and
Li and Zhao~\cite{Li-Zhao:2002,Li-Zhao:2002b,Li-Zhao:2004}.

\begin{theorem}[Invariance of $RG$-measures, Theorem~4 in \cite{ZLB:2003}] \label{theorem1}
  Let $\{ R_{i,j} \}_{j > i \geq 0}$ and $\{ G_{i,j} \}_{0 \leq j < i}$ be the $R$- and $G$-measures defined for the Markov chain $P$ given in \eqref{eqn:Pij}, and let $\{ R_{i,j}^{(n)} \}_{n \geq j > i \geq 0}$ and $\{ G_{i,j}^{(n)} \}_{0 \leq j < i \leq n}$ be the $R$- and $G$-measures defined for the censored Markov chain $P^{(n)}$. Then, for given $j > i \geq 0$,
\begin{equation} \label{Eqn:13}
    R^{(n)}_{i,j} = R_{i,j}, \quad \text{for all $n \geq j$},
\end{equation}
and for given $i > j \geq 0$,
\begin{equation}
    G^{(n)}_{i,j} = G_{i,j}, \quad \text{for all $n \geq i$.}
\end{equation}
\end{theorem}

The $RG$-factorization of $I-P$ for Markov chains of $GI/G/1$ type (or with shifted repeating block rows except the boundary row) was given in \cite{Zhao:2000}, which can be extended to a general block-partitioned Markov chain $P$ defined in \eqref{eqn:Pij}.

\begin{theorem} [$RG$-factorization for general $P$] \label{thm:RG-F}
The $RG$-factorization can be extended to the general block-partitioned Markov chain $P$ given in \eqref{eqn:Pij}, as follows:
\begin{equation} \label{eqn:RG-F}
    I - P = (I - R_U)(I - \Phi_D)(I - G_L),
\end{equation}
where
\[ R_U = \begin{pmatrix}
0 & R_{0,1} & R_{0,2} & R_{0,3} & \cdots \\
  & 0 & R_{1,2} & R_{1,3} & \cdots \\
  &   & 0 & R_{2,3} & \cdots \\
  &   &   & 0 & \cdots \\
  & & & & \ddots
\end{pmatrix}, \]
\[
    \Phi_D = \mathrm{diag}(\Phi_0, \Phi_1, \Phi_2, \ldots) \quad \text{with $\Phi_k = P^{(k)}_{k,k}$,}
\]
and
\[ G_L = \begin{pmatrix}
0 & & & & \\
G_{1,0} & 0 & & & \\
G_{2,0} & G_{2,1} & 0 & & \\
G_{3,0} & G_{3,1} & G_{3,2} & 0 & \\
\vdots & \vdots & \vdots & \vdots & \ddots
\end{pmatrix}.
\]
By convention, all empty entries are zero entries.
\end{theorem}

\begin{proof}
The extended $RG$-factorization can be proved similarly to the derivation in \cite{Zhao:2000} (see also \cite{ZLB:2003}). Specifically,
\begin{eqnarray*}
    P^{(n)}_{n-i,n} &=& P^{(n+1)}_{n-i,n} + P^{(n+1)}_{n-i,n+1} (I - P^{(n+1)}_{n+1,n+1} )^{-1} P^{(n+1)}_{n+1,n} \\
     &=& P^{(n+1)}_{n-i,n} + P^{(n+1)}_{n-i,n+1} (I - P^{(n+1)}_{n+1,n+1} )^{-1} (I - P^{(n+1)}_{n+1,n+1} ) (I - P^{(n+1)}_{n+1,n+1} )^{-1}P^{(n+1)}_{n+1,n} \\
     &=& P^{(n+1)}_{n-i,n} + R^{(n+1)}_{n-i,n+1} (I - P^{(n+1)}_{n+1,n+1} ) G^{(n+1)}_{n+1,n} \\
    &=& P^{(n+2)}_{n-i,n} + R^{(n+2)}_{n-i,n+2} (I - P^{(n+2)}_{n+2,n+2} ) G^{(n+2)}_{n+2,n}  + R^{(n+1)}_{n-i,n+1} (I - P^{(n+1)}_{n+1,n+1} ) G^{(n+1)}_{n+1,n} \\
    &=& \cdots \cdots \\
    &=& P^{(n+K)}_{n-i,n} + \sum_{k=1}^K R^{(n+k)}_{n-i,n+k} (I - P^{(n+k)}_{n+k,n+k} ) G^{(n+k)}_{n+k,n}.
\end{eqnarray*}
Using the invariance under censoring for the $R$- and $G$-measures and taking the limit of $K \to \infty$, we have
\[
    P^{(n)}_{n-i,n} = P_{n-i,n} + \sum_{k=1}^\infty R_{n-i,n+k} (I - \Phi_{n+k} ) G_{n+k,n},
\]
or
\begin{equation} \label{eqn:R}
    R_{n-i,n} (I - \Phi_{n})= P_{n-i,n} + \sum_{k=1}^\infty R_{n-i,n+k} (I - \Phi_{n+k} ) G_{n+k,n}.
\end{equation}
In the same fashion, we have
\begin{equation} \label{eqn:G}
   (I - \Phi_{n})  G_{n,n-i}  = P_{n,n-i} + \sum_{k=1}^\infty R_{n,n+k} (I - \Phi_{n+k} ) G_{n+k,n-i}.
\end{equation}
\eqref{eqn:R} and \eqref{eqn:G} are exactly the same as the detailed equations from the $RG$-factorization in \eqref{eqn:RG-F} by noticing that
\[
    \Phi_0 = P^{(0)} = P_{0,0} + \sum_{k=1}^\infty R_{0,k} (I - \Phi_{k} ) G_{k,0}.
\]
This completes the proof.
\end{proof}

When the Markov chain $P$ consists of finite levels as given in \eqref{eqn:Pij-N}, we have the following corollary.
\begin{corollary} \label{cor:RG-F-N}
For the transition matrix $P$ in \eqref{eqn:Pij-N}, the $RG$-factorization becomes
\begin{equation} \label{eqn:RG-F-N}
    I - P = (I - R_U)(I - \Phi_D)(I - G_L),
\end{equation}
where
\[ R_U = \begin{pmatrix}
0 & R_{0,1} & R_{0,2} & R_{0,3} & \cdots & R_{0,N} \\
  & 0 & R_{1,2} & R_{1,3} & \cdots & R_{1,N} \\
  &   & 0 & R_{2,3} & \cdots & R_{2,N} \\
  &   &   & \ddots & \vdots & \vdots \\
  &   &   &        & \ddots & R_{N-1,N} \\
  & & & &  & 0
\end{pmatrix}, \]
\[
    \Phi_D = \operatorname{diag}(\Phi_0, \Phi_1, \ldots, \Phi_N) \quad \text{with $\Phi_k = P^{(k)}_{k,k}$,}
\]
and
\[ G_L = \begin{pmatrix}
0 & & & & \\
G_{1,0} & 0 & & & \\
G_{2,0} & G_{2,1} & 0 & & \\
G_{3,0} & G_{3,1} & G_{3,2} & 0 & \\
\vdots  & \vdots & \vdots & \vdots & \ddots \\
G_{N,0} & G_{N,1} & G_{N,2} & \cdots & G_{N,N-1} & 0
\end{pmatrix}. \]
\end{corollary}

We are now ready to connect the block-form GTH algorithm with the $RG$-factorization.

\begin{lemma}[Block–form GTH implies $RG$-factorization]
\label{lem:blockGTH-to-RG}
Let $P$ be an irreducible block–form transition matrix given by \eqref{eqn:Pij-N}.
Running the forward block–elimination of the block-form GTH algorithm gives the $RG$-factorization given in
\eqref{eqn:RG-F-N}.
\end{lemma}

\begin{proof}
Based on the discussion in Section~\ref{sec:censoring}, we know the forward elimination is the same as censoring, or
\[
    P^{\langle n \rangle} = P^{(n)},
\]
where $P^{\langle n \rangle}$ is the matrix from the forward elimination and $P^{(n)}$ is the censored matrix with censoring set $L_{\leq n}$. Therefore, we have
\[
P^{\langle N\rangle}:=P,\qquad
P^{\langle k-1\rangle}_{i,j}
\;=\;
P^{\langle k\rangle}_{i,j}
\;+\;
P^{\langle k\rangle}_{i,k}\,\bigl(I-P^{\langle k\rangle}_{k,k}\bigr)^{-1}P^{\langle k\rangle}_{k,j},
\quad 0\le i,j\le k-1,\; k=N,\dots,1.
\]
For each eliminated level $k$, it follows from $P^{\langle n \rangle} = P^{(n)}$ and the invariant property of $R$- and $G$-measures under censoring that
\[
R_{i,k} = P^{\langle k\rangle}_{i,k}\,\bigl(I-P^{\langle k\rangle}_{k,k}\bigr)^{-1}
\quad(0\le i<k),\qquad
G_{k,i} = \bigl(I-P^{\langle k\rangle}_{k,k}\bigr)^{-1}P^{\langle k\rangle}_{k,i}
\quad(0\le i<k).
\]
Moreover, define
\[
\Phi_D\ :=\ \mathrm{diag}\bigl(\Phi_0,\Phi_1,\dots,\Phi_N\bigr),
\qquad \Phi_k:=P^{\langle k\rangle}_{k,k}.
\]
Then, by Theorem~\ref{thm:RG-F} (specialized to the finite-level case), the $RG$-factorization holds.
\end{proof}

\paragraph{Recovering forward elimination and back substitution from the $RG$-factorization.} Next, we
recover the forward block-elimination and back block-substitution via the $RG$-factorization in \eqref{eqn:RG-F-N}.
Let $\pi=[\pi_0,\pi_1,\ldots,\pi_N]$ be the stationary row vector (to be normalized at the end). Define, by denoting $R=R_U$, $G=G_L$ and $\Phi = \Phi_D$,
\begin{equation}\label{eq:y-z-def}
y \;:=\; \pi\,(I-R),
\qquad
z \;:=\; y\,(I-\Phi).
\end{equation}
Then, the stationary equation $\pi(I-P)=0$ is equivalent to the \emph{triangular left–null} equation
\begin{equation}\label{eq:triangular-null}
z\,(I-G) \;=\; 0 .
\end{equation}

\paragraph{Forward elimination (level–wise consequences of \eqref{eq:triangular-null}).}
Since $I-G$ is unit lower block–triangular, \eqref{eq:triangular-null} forces the upper blocks of $z$ to vanish:
\begin{equation}\label{eq:z-upper-zero}
z_j \;=\; 0 \quad\text{for all } j\ge 1,
\qquad\text{and } z_0 \text{ is free.}
\end{equation}
Because $z_j = y_j\,(I-\Phi_j)$ and $I-\Phi_j$ is invertible for $j\ge 1$, we immediately obtain
\begin{equation}\label{eq:y-upper-zero}
y_j \;=\; 0 \quad\text{for all } j\ge 1,
\qquad\text{while } z_0 = y_0\,(I-\Phi_0) \in \ker_{\mathrm{left}}(I-\Phi_0).
\end{equation}
Thus $y=[\,y_0,\,0,\ldots,0\,]$ with $y_0$ a (nonzero) multiple of the stationary row of $\Phi_0$.

\paragraph{Back substitution (recovering $\pi$ from $y=\pi(I-R)$).}
Because $I-R$ is unit upper block–triangular, $y=\pi(I-R)$ reads, blockwise,
\begin{equation}\label{eq:upper-tri-system}
y_0=\pi_0,
\qquad
y_j \;=\; \pi_j \;-\; \sum_{i<j}\pi_i\,R_{i,j}\quad(j=1,\ldots,N).
\end{equation}
Using \eqref{eq:y-upper-zero}, we get $y_j=0$ for all $j\ge 1$, hence the \emph{backward recursion}
\begin{equation}\label{eq:back-recursion}
\pi_j \;=\; \sum_{i<j}\pi_i\,R_{i,j}\qquad (j=1,\ldots,N),
\end{equation}
with the level–$0$ seed $\pi_0=y_0$ chosen as the stationary row of $\Phi_0$ (up to a common positive scalar).
Finally we normalize:
\begin{equation}\label{eq:normalize}
\pi \;\leftarrow\; \frac{\pi}{\sum_{k=0}^N \mathbf{1}^\top \pi_k}\,.
\end{equation}

\begin{remark}
The equivalence between the GTH algorithm and $RG$-factorization can be extended formally for the infinite transition matrix in \eqref{eqn:Pij}.
\end{remark}

\section{Censored and approximated censored matrices for Markov chains of $M/G/1$ type}
\label{sec:RA-CM}

The second part of this paper starts with this section. According to the discussions in the first part, it is clear that the stationary distribution of the censored chain in \eqref{eqn:censored-from-infinite} gives a best approximation to the stationary probability distribution of the original (infinite-level) chain among all possible augmentation methods with the same truncation size $N$ in terms of $\ell_1$ norm. However, in general there is no simple or explicit expression for the censored matrix, or $U_N \hat{Q}_N D_N$.
For the scalar case, if the transition matrix is upper-Hessenberg, then it is well known that the censored matrix is simply the same as the last-column augmented matrix. This is no longer true for the block-upper-Hessenberg case:
\begin{equation} \label{eqn:upper-H}
  P=\begin{pmatrix}
P_{0,0} & P_{0,1} & P_{0,2} & P_{0,3} & P_{0,4} & \cdots  \\
P_{1,0} & P_{1,1} & P_{1,2} & P_{1,3} & P_{1,4} & \cdots  \\
        & P_{2,1} & P_{2,2} & P_{2,3} & P_{2,4} & \cdots  \\
        &         & P_{3,2} & P_{3,3} & P_{3,4} & \cdots \\
& & &\ddots & \ddots&  \vdots \\
& & & & \ddots & \ddots
\end{pmatrix} = \bordermatrix[{[]}]{%
    & L_{\leq N} &L^c_{\leq N} \cr
L_{\leq N} & T_N & U_N \cr
L^c_{\leq N} & D_N & Q
},
\end{equation}
for which the censored matrix $P^{(N)}$ is in general not the same as
\[
    T_N + (0, \ldots,0, C_N),
\]
where
\[
    C_N = \begin{pmatrix} \sum_{k=N+1}^{\infty} P_{0,k} \\ \vdots \\ \sum_{k=N+1}^{\infty} P_{N,k}
    \end{pmatrix}.
\]
This is true even when $P$ is the matrix of $M/G/1$ type:
\begin{equation} \label{eqn:MG1-type}
  P=\begin{pmatrix}
B_0 & B_1 & B_2 & B_3 & B_4 & \cdots   \\
A_{-1}&A_0 & A_1 & A_2 & A_3 & \cdots  \\
 & A_{-1} & A_{0} & A_{1} & A_2 & \cdots  \\
 &  & A_{-1} & A_0 & A_1  & \cdots \\
& & &\ddots & \ddots&  \vdots \\
& & & & \ddots & \ddots
\end{pmatrix}.
\end{equation}

In this section, we first obtain an exact expression for the censored matrix of $M/G/1$ type, based on which an approximated (or truncated) censored (sub-stochastic) matrix can computed efficiently. We then show that (1) the renormalized approximated censored transition matrix (RA-CM) is positive recurrent when the original is positive recurrent; and (2) the stationary probability distribution of RA-CM is an approximation of the probability distribution of the original Markov chain, which is asymptotically best in approximation error in the sense of $\ell_1$ norm.

Throughout this section, we assume that the original Markov chain, whose transition matrix is given in \eqref{eqn:MG1-type}, is positive recurrent with stationary distribution $\pi=(\pi_0, \pi_1, \pi_2, \ldots)$, partitioned according to levels too, where
\[
    \pi_n = (\pi_{n,1}, \pi_{n,2}, \ldots, \pi_{n,r}).
\]

\subsection{Exact expression of the censored matrix}

Partition the transition probability matrix $P$ in \eqref{eqn:MG1-type} according to $L_{\leq N}$ and $L^c_{\leq N}$:
\begin{equation} \label{eqn:MG1-type-b}
    P  = \bordermatrix[{[]}]{%
    & L_{\leq N} &L^c_{\leq N} \cr
L_{\leq N} & T_N & U_N \cr
L^c_{\leq N} & D_N & Q
},
\end{equation}
where
\begin{equation*}
  T_N =\begin{pmatrix}
B_0 & B_1 & B_2 & \cdots & \cdots & B_N   \\
A_{-1}&A_0 & A_1 & \cdots & \cdots & A_{N-1}  \\
 & A_{-1} & A_{0} &  \cdots & \cdots & A_{N-2} \\
 & &\ddots & \ddots& \cdots & \vdots \\
  & & & A_{-1} & A_0 & A_1 \\
 & & & & A_{-1} & A_0
\end{pmatrix},
\end{equation*}

\begin{equation*}
  U_N =\begin{pmatrix}
B_{N+1} & B_{N+2} & B_{N+3} & \cdots & \cdots &   \\
A_{N} &  A_{N+1} & A_{N+2} & \cdots & \cdots  \\
A_{N-1} &  A_{N} & A_{N+1} & \cdots & \cdots  \\
\vdots & \vdots & \vdots&  \vdots & \vdots \\
A_{1} &  A_{2} & A_{3} & \cdots & \cdots  \\
\end{pmatrix},
\end{equation*}

\begin{equation*}%\label{}
  D_N=\begin{pmatrix}
0 & \cdots & \cdots & 0 & A_{-1}  \\
0 & \cdots & \cdots & 0 & 0  \\
\vdots & \vdots & \vdots & \vdots & \vdots
\end{pmatrix},
\end{equation*}
and
\begin{equation*}%\label{}
  Q=\begin{pmatrix}
A_0 & A_1 & A_2 & A_3 &  \cdots  \\
A_{-1} & A_{0} & A_{1} & A_2 & \cdots  \\
 & A_{-1} & A_0 & A_1  & \cdots \\
  & & A_{-1} & A_0 & \cdots \\
& & &\ddots & \ddots
\end{pmatrix},
\end{equation*}
which is independent of $N$.

Now, consider the censored Markov chain of $P$ in \eqref{eqn:MG1-type} with censoring set $L_{\leq N}$, whose transition probability matrix is given by
\[
   P^{(N)} = T_N + U_N \hat{Q} D_N,
\]
where $\hat{Q} = \sum_{n=0}^{\infty} Q^n$ is the fundamental matrix of $Q$.

Let $\pi^{(N)}=(\pi^{(N)}_0, \pi^{(N)}_1, \ldots, \pi^{(N)}_N)$ be the stationary probability vector for the censored process $P^{(N)}$. Then, we have the following property:
\[
    \pi^{(N)}_n = \frac{\pi_n}{\sum_{i=0}^{N} \sum_{\alpha=1}^{r} \pi_{i,\alpha}}, \quad n=0, 1, \ldots, N.
\]

\begin{remark}[slice notation ``:'']
% ---- Indexing convention for block matrices ----
For a block matrix $M$ with block-rows/columns indexed by levels: \\
\hspace*{1.5cm}  Block-row $i$:  $M_{i,\,:}$ \\
\hspace*{1.5cm}  Block-column $j$:  $M_{:,\,j}$ \\
\hspace*{1.5cm}  Single scalar row $(i,\alpha)$ over phases:  $[M]_{(i,\alpha),\,:}$ \\
\hspace*{1.5cm}  Single scalar column $(j,\beta)$ over phases: $[M]_{:,\, (j,\beta)}$ \\

\noindent \textbf{Examples:}

Last block-column of $Q^m D$:
$(Q^m D)_{:,\,N}$

Block-row $i$ of $U$ times last block-column of $Q^m D$:
$U_{i,\, :}\,(Q^m D)_{:,\,N}$

Total mass in scalar row $(i,\alpha)$ into level $N$ (in last block-column of $\hat Q^{(M)}D)$):
$\|[\,(\hat Q^{(M)}D)_{:,\,N}\,]_{(i,\alpha),\,:}\|_1$
\end{remark}

Since $\hat{Q}$ is independent of $N$, and for any $N \geq 0$, all block-columns, except the last one, of $D_N$ are zero, we have that all block-columns, except the last one, of $\hat{Q} D_N$ are zero too. In the following, without confusion, we write $D=D_N$. Recall that
\[
    \hat{Q} D = \big ( \sum_{m=0}^{\infty} Q^m \big ) D = \sum_{m=0}^{\infty} (Q^m D).
\]
We now recursively calculate the last block-column of $Q^m D$, denoted by $(Q^m D)_{:,\, N}$, for $m=0, 1, 2, 3, 4$, based on which we identify the product-summary pattern of matrices in the next theorem.

\textbf{For $m=0$:}
\[
(Q^0 D)_{:,\, N}
=
\begin{pmatrix}
A_{-1}\\
0\\
\vdots
\end{pmatrix}.
\]

\textbf{For $m=1$:}
\[
(Q D)_{:,\, N}
=
\begin{pmatrix}
A_{0}A_{-1}\\
A_{-1}A_{-1}\\
0\\
\vdots
\end{pmatrix}.
\]

\textbf{For $m=2$:}
\[
(Q^2D)_{:,\, N}
=
\begin{pmatrix}
\bigl(A_{0}^{2}+A_{1}A_{-1}\bigr)A_{-1}\\
\bigl(A_{-1}A_{0}+A_{0}A_{-1}\bigr)A_{-1}\\
A_{-1}^{3}\\
0\\
\vdots
\end{pmatrix}.
\]

\textbf{For $m=3$:}
\[
(Q^3D)_{:,\, N}
=
\begin{pmatrix}
\bigl(A_{0}^{3}+A_{0}A_{1}A_{-1}+A_{1}A_{-1}A_{0}+A_{1}A_{0}A_{-1}+A_{2}A_{-1}^{2}\bigr)A_{-1}\\[4pt]
\bigl(A_{-1}A_{0}^{2}+A_{-1}A_{1}A_{-1}+A_{0}A_{-1}A_{0}+A_{0}^{2}A_{-1}+A_{1}A_{-1}^{2}\bigr)A_{-1}\\[4pt]
\bigl(A_{-1}^{2}A_{0}+A_{-1}A_{0}A_{-1}+A_{0}A_{-1}^{2}\bigr)A_{-1}\\[2pt]
A_{-1}^{4}\\
0\\
\vdots
\end{pmatrix}.
\]

\textbf{For $m=4$:}
\[
(Q^4D)_{:,\, N}
=
\begin{pmatrix}
X_0^{(4)}\\[4pt]
X_1^{(4)}\\[4pt]
X_2^{(4)}\\[4pt]
X_3^{(4)}\\[4pt]
X_4^{(4)}\\[2pt]
0\\
\vdots
\end{pmatrix},
\]
where
\[
\begin{aligned}
X_0^{(4)} \!&=\! \Bigl(
   A_0^4
 + A_0^2A_1A_{-1}
 + A_0A_1A_{-1}A_0
 + A_0A_1A_0A_{-1}
 + A_0A_2A_{-1}^2\\
&\qquad
 + A_1A_{-1}A_0^2
 + A_1A_{-1}A_1A_{-1}
 + A_1A_0A_{-1}A_0
 + A_1A_0^2A_{-1}
 + A_1^2A_{-1}^2\\
&\qquad
 + A_2A_{-1}^2A_0
 + A_2A_{-1}A_0A_{-1}
 + A_2A_0A_{-1}^2
 + A_3A_{-1}^3
\Bigr)A_{-1},\\[6pt]
X_1^{(4)} \!&=\! \Bigl(
   A_{-1}A_0^3
 + A_{-1}A_0A_1A_{-1}
 + A_{-1}A_1A_{-1}A_0
 + A_{-1}A_1A_0A_{-1}
 + A_{-1}A_2A_{-1}^2\\
&\qquad
 + A_0A_{-1}A_0^2
 + A_0A_{-1}A_1A_{-1}
 + A_0^2A_{-1}A_0
 + A_0^3A_{-1}
 + A_0A_1A_{-1}^2\\
&\qquad
 + A_1A_{-1}^2A_0
 + A_1A_{-1}A_0A_{-1}
 + A_1A_0A_{-1}^2
 + A_2A_{-1}^3
\Bigr)A_{-1},\\[6pt]
X_2^{(4)} \!&=\! \Bigl(
   A_{-1}^2A_0^2
 + A_{-1}^2A_1A_{-1}
 + A_{-1}A_0A_{-1}A_0
 + A_{-1}A_0^2A_{-1}
 + A_{-1}A_1A_{-1}^2\\
&\qquad
 + A_0A_{-1}^2A_0
 + A_0A_{-1}A_0A_{-1}
 + A_0^2A_{-1}^2
 + A_1A_{-1}^3
\Bigr)A_{-1},\\[6pt]
X_3^{(4)} \!&=\! \Bigl(
   A_{-1}^3A_0
 + A_{-1}^2A_0A_{-1}
 + A_{-1}A_0A_{-1}^2
 + A_0A_{-1}^3
\Bigr)A_{-1},\\[6pt]
X_4^{(4)} \!&=\; A_{-1}^5.
\end{aligned}
\]
% Blue terms are the additions/fixes relative to the previous draft.

% Cleaned version: avoid duplicate notation for s_m by using level variables \ell_r

\begin{theorem}[Constrained path expansion for \(Q^mD\)]
\label{thm:QmD-censored-clean}
Let \(P\) be the transition matrix of the \(M/G/1\) type Markov chain, given in \eqref{eqn:MG1-type} and partitioned according to censoring set $L_{\leq N}$ as in \eqref{eqn:MG1-type-b},
where \(Q\) acts on the complement levels \(\{N+1,N+2,\dots\}\), which we renumber as \(0,1,2,\dots\). Level for matrices $U:=U_N$ and $D:=D_N$ are renumbered accordingly. Recall that \(Q_{i,t}=A_{t-i}\) with \(A_k=0\) for \(k<-1\), and that \(D\) has the single nonzero block \((D)_{0,N}=A_{-1}\).

Then, for $m=0$, we have $(Q^0 D) = D$ and therefore,
\[
    (Q^0 D)_{i,N}=\mathbf{1}_{\{i=0\}}\,A_{-1}.
\]
For integers \(m\ge 1\) and \(i\ge 0\), let
\[
    \ell_0=i,\qquad \ell_u = i+\sum_{j=1}^u k_j \ (1\le u\le m),
\]
and define the admissible step-set
\[
\mathcal{K}_m(i)
= \Bigl\{ (k_1,\dots,k_m)\in\{-1,0,1,2,\dots\}^m \ \Bigm|\
\ell_u \ge 0\ (0\le u\le m),\ \ell_m=0 \Bigr\}.
\]
Equivalently, the path starts at level \(i\), never goes below \(0\), and ends at \(0\). In particular, \(\sum_{j=1}^m k_j=-i\). Then,

\begin{equation}\label{eq:QmD-clean}
(Q^m D)_{i,N}
\;=\;
\sum_{(k_1,\dots,k_m)\in\mathcal{K}_m(i)} A_{k_1}\cdots A_{k_m}\,A_{-1},
\end{equation}
with \((Q^mD)_{i,N}=0\) whenever \(i>m\).

Consequently,
\[
(\hat Q D)_{i,N}
=\sum_{m=0}^{\infty}(Q^mD)_{i,N}
=\mathbf{1}_{\{i=0\}}\,A_{-1}
+ \sum_{m=1}^{\infty}\ \sum_{(k_1,\dots,k_m)\in\mathcal{K}_m(i)} A_{k_1}\cdots A_{k_m}\,A_{-1},
\qquad i\ge 0,
\]
and the censored transition matrix $P^{(N)}$ can be explicitly expression according to  \(\ P^{(N)}=T_N+U_N\hat Q D_N\).
\end{theorem}

\begin{proof}
The result for $m=0$ is obvious. We prove the theorem by induction on \(m \geq 1\).

\emph{Base step \(m=1\).}
Using block multiplication and that \(D_{t,N}=0\) unless \(t=0\),
\[
(QD)_{i,N}=\sum_{t\ge 0}Q_{i,t}D_{t,N}=Q_{i,0}\,A_{-1}=A_{0-i}\,A_{-1}.
\]
Because \(A_k=0\) for \(k<-1\), this equals \(A_0A_{-1}\) when \(i=0\), \(A_{-1}A_{-1}\) when \(i=1\), and \(0\) for \(i\ge 2\).
On the other hand,
\[
\mathcal K_1(i)=\{(k_1): k_1\in\{-1,0,1,\dots\},\ i+k_1=0\},
\]
so \(\mathcal K_1(0)=\{(0)\}\), \(\mathcal K_1(1)=\{(-1)\}\), and \(\mathcal K_1(i)=\emptyset\) for \(i\ge 2\).
Hence
\[
(QD)_{i,N}=\sum_{(k_1)\in\mathcal K_1(i)} A_{k_1}A_{-1},
\]
which is exactly \eqref{eq:QmD-clean} for \(m=1\).

\textbf{Inductive step.}
Assume \eqref{eq:QmD-clean} holds for some \(m\ge 1\); we show it for \(m+1\).
Using \(Q_{i,t}=A_{t-i}\) and the inductive hypothesis at level \(m\),
\[
\begin{aligned}
(Q^{m+1}D)_{i,N}
&=\sum_{t\ge 0} Q_{i,t}\,(Q^mD)_{t,N}
=\sum_{t\ge 0} A_{t-i}\!\!\sum_{(k_2,\dots,k_{m+1})\in\mathcal K_m(t)} \!\!\!A_{k_2}\cdots A_{k_{m+1}}A_{-1} \\
&=\sum_{\substack{k_1\ge -1\\ i+k_1\ge 0}}
\ \sum_{(k_2,\dots,k_{m+1})\in\mathcal K_m(i+k_1)}
A_{k_1}\cdots A_{k_{m+1}}A_{-1}.
\end{aligned}
\]
But concatenating the first step \(k_1\) with any admissible \(m\)-step path from level \(i+k_1\) to \(0\) gives precisely the admissible
\((m\!+\!1)\)-step paths from \(i\) to \(0\); conversely, every element of \(\mathcal K_{m+1}(i)\) arises uniquely this way.
Therefore the double sum reduces to
\[
(Q^{m+1}D)_{i,N}
=\sum_{(k_1,\dots,k_{m+1})\in\mathcal K_{m+1}(i)} A_{k_1}\cdots A_{k_{m+1}}A_{-1},
\]
which is \eqref{eq:QmD-clean} for \(m+1\).

\textbf{Vanishing for \(i>m\).}
If \(i>m\) then any \(m\)-step sequence satisfies \(\sum_{j=1}^m k_j\ge -m\) since \(k_j\ge -1\), so it cannot reach \(-i\).
Thus \(\mathcal K_m(i)=\emptyset\) and \((Q^mD)_{i,N}=0\).
\end{proof}

%\begin{remark}[Why \(k\ge -1\)]
%Each block \(A_k\) encodes a one-step level change of \(+k\). In \(M/G/1\)-type chains, a step can drop by at most
%one level (service completion) or increase by a nonnegative amount (arrivals), hence \(A_k=0\) for \(k<-1\).
%This enforces \(k_r\ge -1\) in \(\mathcal{K}_m(i)\) and ensures the first step from \(i\) reaches a nonnegative
%intermediate level \(t=i+k_1\ge 0\).
%\end{remark}

%\begin{example}[Check of \((Q^2D)_{0,N}\)]
%From the recursion,
%\((QD)_{0,N}=A_0A_{-1}\), \((QD)_{1,N}=A_{-1}A_{-1}\), and \((QD)_{t,N}=0\) for \(t\ge 2\).
%Thus
%\[
%(Q^2D)_{0,N}
%=Q_{0,0}(QD)_{0,N}+Q_{0,1}(QD)_{1,N}
%=A_0(A_0A_{-1})+A_1(A_{-1}A_{-1})
%=A_0A_0A_{-1}+A_1A_{-1}A_{-1}.
%\]
%Our admissible set gives \(\mathcal{K}_2(0)=\{(0,0),(1,-1)\}\);
%the candidate \((-1,1)\) is invalid because the intermediate level would be \(\ell_1=-1<0\).
%\end{example}

\begin{corollary}
For irreducible positive recurrent and aperiodic Markov chains of $M/G/1$ type with its transition matrix given in \eqref{eqn:MG1-type} and partitioned as in \eqref{eqn:MG1-type-b}, define
\[
   G_0^{(m)} \;\stackrel{\triangle}{=}\;(Q^mD)_{0,N}.
\]
Then,
 the (fundamental-period) matrix $G$ can be explicitly expressed as \(
    G =\sum_{m=0}^{\infty} G_0^{(m)}.
\)
\end{corollary}

\proof It follows directly from Lemma~4 of \cite{Zhao:2000}. \QED

\subsection{Approximating the censored Markov chain for Markov chain of $M/G/1$ type}

We cannot directly use the explicit expression to compute the censored matrix $P^{(N)}$ since the fundamental matrix $\hat{Q}$ consists of a sum of infinite terms. A natural approximation for the censored matrix is
\[
   P^{(N;M)} \; := \; T_N \;+\; U(\hat Q^{(M)})D,
\]
where
\[
    \hat Q^{(M)} := \sum_{m=0}^{M} Q^m.
\]
$P^{(N;M)}$ is referred to as an approximated (or truncated) censored matrix.

In the following, we provide an error bound based on which we can determine a stop point (or $M$) such that the error is bounded by any arbitrarily given $\varepsilon>0$.

For $m\ge0$ define
\begin{equation}\label{eqn:truncated}
    G^{(m)} := (Q^mD)_{:,\,N} =
\begin{pmatrix}
G^{(m)}_0\\[4pt]
G^{(m)}_1\\[4pt]
G^{(m)}_2\\[4pt]
\vdots
\end{pmatrix},
\end{equation}
where block $G^{(m)}_s$ is an $r\times r$ nonnegative matrix.
Then,
\[
    \hat Q^{(M)} D = \sum_{m=0}^M G^{(m)}.
\]

% ----------------------------------------------------------------
% Theorem: error bound for M--truncation of \hat Q D and of P^{L_{\le N}}
% ----------------------------------------------------------------
\begin{theorem}[Error bound for $M$--truncation]
\label{thm:truncation-error}
Define the truncation error column
\[
E^{(M)}
\;:=\; (\hat Q D)_{:,\, N} - (\hat Q^{(M)} D)_{:,\, N}
\;=\; \sum_{m=M+1}^\infty G^{(m)}.
\]
Let $\Delta^{(M)} := U\,E^{(M)} = U\bigl(\hat Q D - \hat Q^{(M)} D \bigr)$ be the error matrix between the censored matrix and the approximated censored matrix, which is a matrix of $(N+1) \times (N+1)$ blocks,.

%Write each block $G^{(m)}_s$ (the block of $G^{(m)}$ corresponding to the source level $s$ in $L_{\le N}^c$)
%as an $r\times r$ nonnegative matrix.
We use indices $(i,\alpha)$ for an origin state in $L_{\le N}$ (with $0\le i\le N$ and phase $\alpha\in\{1,\dots,r\}$),
and $(s,p)$ for a source state in the complement $L_{\le N}^c$ (with $s\ge 0$ and phase $p\in\{1,\dots,r\}$).
For each $s\ge 0$ and $p\in\{1,\dots,r\}$ define the \emph{captured return mass}
\[
r_{s,p}^{(M)} \;:=\; \sum_{m=0}^M \sum_{d=1}^r \bigl[G^{(m)}_s\bigr]_{p,d},
\]
i.e., the total probability (starting from $(s,p)$) of returning to the $N$th-column target
after spending at most $M$ steps in $L_{\le N}^c$.
Then the following hold.

%\begin{enumerate}[\upshape(a)]
\begin{enumerate}[label=\textup{(\alph*)}]
\item \emph{Entrywise representation and nonnegativity.}
For every source level $s\ge 0$,
\[
E^{(M)}_s \;=\; \sum_{m=M+1}^\infty G^{(m)}_s,
\]
hence all entries of $E^{(M)}$ and $\Delta^{(M)}$ are nonnegative.

\item \emph{Computable a posteriori bound.}
For any origin state $(i,\alpha)$ with $0\le i\le N$ and any destination phase $\beta\in\{1,\dots,r\}$ in the $N$th block,
\[
    \bigl[\Delta^{(M)}_{i,N}\bigr]_{\alpha,\beta}
\;\le\;
\sum_{s\ge 0}\sum_{p=1}^r U_{i,s}(\alpha,p)\;\bigl(1-r_{s,p}^{(M)}\bigr),
\]
where $U_{i,s}(\alpha,p)$ denotes the $(\alpha,p)$ entry of the block $U_{i,s}$.
All quantities on the right-hand side are computable from the truncated iterates $G^{(0)},\dots,G^{(M)}$.

\item \emph{Uniform stopping criterion.}
Since for any fixed origin state $(i,\alpha)$ one has
\(\sum_{s\ge 0}\sum_{p=1}^r U_{i,s}(\alpha,p)\le 1\) (the probability that the chain jumps from
\((i,\alpha)\) into $L_{\le N}^c$ in one step), the previous display implies the uniform bound
\[
\max_{i,\alpha,\beta} \bigl[\Delta^{(M)}_{i,N}\bigr]_{\alpha,\beta}
\;\le\;
\max_{s\ge 0}\max_{p=1,\dots,r} \bigl(1-r_{s,p}^{(M)}\bigr).
\]
Therefore, if the computed captured masses satisfy
\[
\max_{s\ge 0,\;p}\ r_{s,p}^{(M)} \; \ge \; 1-\varepsilon,
\]
then every scalar entry of the truncation error $\Delta^{(M)}$ is bounded by $\varepsilon$.
In particular, this gives a straightforward, fully computable stopping rule: stop the recursion when
\(\max_{s\ge 0,\;p}\bigl(1-r_{s,p}^{(M)}\bigr)\le\varepsilon\).
\end{enumerate}
\end{theorem}

\begin{proof}
Recall $G^{(m)}=(Q^mD)_{:,\,N}$, $\hat Q^{(M)} D =\sum_{m=0}^M G^{(m)}$,
$E^{(M)}=\sum_{m=M+1}^\infty G^{(m)}$, and $\Delta^{(M)}=U\,E^{(M)}$.
All blocks are entrywise nonnegative, and all series are termwise nondecreasing, hence (by monotone convergence)
well-defined entrywise.

\smallskip
\noindent\emph{(a) Entrywise representation and nonnegativity.}
By definition of $E^{(M)}$, taking the block at any fixed source level $s\ge 0$ gives
$E^{(M)}_s=\sum_{m=M+1}^\infty G^{(m)}_s$.
Since every $G^{(m)}_s$ is entrywise nonnegative, the same holds for $E^{(M)}_s$ and thus for
$\Delta^{(M)}=U\,E^{(M)}$.

\smallskip
\noindent\emph{(b) Computable a posteriori bound.}
Fix an origin state $(i,\alpha)$ with $0\le i\le N$ and a destination phase $\beta\in\{1,\dots,r\}$ in the $N$th block. Then,
\[
\bigl[\Delta^{(M)}_{i,N}\bigr]_{\alpha,\beta}
=\sum_{s\ge 0}\sum_{p=1}^r U_{i,s}(\alpha,p)\,\bigl[E^{(M)}_s\bigr]_{p,\beta}
\;\le\;
\sum_{s\ge 0}\sum_{p=1}^r U_{i,s}(\alpha,p)\,
\Bigl(\sum_{d=1}^r \bigl[E^{(M)}_s\bigr]_{p,d}\Bigr),
\]
where we bounded a fixed entry of the nonnegative row by its row sum.
Using $E^{(M)}_s=\sum_{m=M+1}^\infty G^{(m)}_s$ and exchanging the finite sum over $d$ with the series,
\[
\sum_{d=1}^r \bigl[E^{(M)}_s\bigr]_{p,d}
=\sum_{m=M+1}^\infty \sum_{d=1}^r \bigl[G^{(m)}_s\bigr]_{p,d}
=\Bigl(\sum_{m=0}^\infty \sum_{d=1}^r \bigl[G^{(m)}_s\bigr]_{p,d}\Bigr)
-\sum_{m=0}^{M} \sum_{d=1}^r \bigl[G^{(m)}_s\bigr]_{p,d}.
\]
By definition $r_{s,p}^{(M)}=\sum_{m=0}^{M} \sum_{d=1}^r [G^{(m)}_s]_{p,d}$,
and the full series $\sum_{m=0}^\infty \sum_{d} [G^{(m)}_s]_{p,d}$ is the (entrywise) probability
that, starting from the source state $(s,p)\in L_{\le N}^c$, one eventually returns to the $N$th-column
target via $D$ after spending some number of steps in $L_{\le N}^c$. Hence, this total probability is
$\le 1$, and we obtain
\[
\sum_{d=1}^r \bigl[E^{(M)}_s\bigr]_{p,d}
\;\le\; 1 - r_{s,p}^{(M)}.
\]
Substituting into the previous display yields
\[
\bigl[\Delta^{(M)}_{i,N}\bigr]_{\alpha,\beta}
\;\le\;
\sum_{s\ge 0}\sum_{p=1}^r U_{i,s}(\alpha,p)\,\bigl(1-r_{s,p}^{(M)}\bigr),
\]
which is the claimed bound.

\smallskip
\noindent\emph{(c) Uniform stopping criterion.}
For fixed $(i,\alpha)$, the row sum $\sum_{s\ge 0}\sum_{p=1}^r U_{i,s}(\alpha,p)$
is the one-step probability that the chain jumps from $(i,\alpha)\in L_{\le N}$ into the complement $L_{\le N}^c$,
hence it is $\le 1$. Therefore,
\[
\bigl[\Delta^{(M)}_{i,N}\bigr]_{\alpha,\beta}
\;\le\;
\max_{s\ge 0}\max_{p=1,\dots,r} \bigl(1-r_{s,p}^{(M)}\bigr)
\sum_{s\ge 0}\sum_{p=1}^r U_{i,s}(\alpha,p)
\;\le\;
\max_{s\ge 0,\;p}\bigl(1-r_{s,p}^{(M)}\bigr).
\]
Taking the maximum over $(i,\alpha,\beta)$ gives the displayed uniform bound.
Consequently, if $\max_{s\ge 0,\;p} r_{s,p}^{(M)} \ge 1-\varepsilon$, then every scalar entry of
$\Delta^{(M)}$ is $\le \varepsilon$, which furnishes the stated stopping rule.
\end{proof}

For Markov chains of $M/G/1$ type,
recall that the transition matrix of the censored Markov chain with censoring set \( L_{\leq N} \) for \( N \geq 0 \) is given by
\[
    P^{(N)} \;=\; T_N \;+\; U\,\hat Q\,D
\quad \text{with} \quad \hat Q \;=\; \sum_{m=0}^{\infty} Q^m,
\]
where all block-entries of \(D\) are zero except the single block
\[
D_{(s=0,\;j=N)} \;=\; A_{-1},
\]
i.e., from the first complement level (\(s=0\)) back to censored level \(N\). Therefore, \(P^{(N)}\) differs from \(T_N\) only in its \emph{last} block-column, which is
\[
    P^{(N)} \;=\; T_N \;+\; (\,0,\,0,\,\dots,\,0,\,C_N^{(\infty)}\,),
    \qquad
    C_N^{(\infty)} \;:=\; \bigl[\,U\,\hat Q\,D\,\bigr]_{:,\, N}.
\]
For a truncation depth \(M\ge 0\), let the \emph{truncated fundamental matrix} be
\[
\hat Q^{(M)} := \sum_{m=0}^{M} Q^m,
\quad\text{and set}\quad
C_N^{(M)} := \bigl[\,U\,(\hat Q^{(M)})D\,\bigr]_{:,\,N}.
\]
Hence, the \(M\)-approximated censored matrix (kernel) on \(L_{\le N}\) is
\[
P^{(N;M)} \;=\; T_N \;+\; (\,0,\,0,\,\dots,\,0,\,C_N^{(M)}\,),
\]
which is generally \emph{sub-stochastic} rowwise.

\subsubsection{Rowwise Renormalization to a Stochastic Matrix}

In general, the approximated (or truncated) censored matrix is only sub-stochastic. A standard method is to ``renormalize" it into a stochastic matrix and to use its stationary probability distribution as an approximation of the distribution of the censored Markov chain. In this section,
we adopt a per-row renormalization scheme that makes the approximated censored matrix exactly stochastic by enforcing each scalar row to sum to one. This corrects the approximation error introduced by truncating the infinite tail.

\paragraph{Target row masses and current tail row masses.}
For each scalar row \((i,\alpha)\) with \(i\in\{0,\dots,N\}\) and \(\alpha=1,\dots,r\), define the target (missing) mass
\[
c_{i,\alpha} \;:=\; 1 - \sum_{j=0}^{N} \sum_{\beta=1}^r \bigl[T_{N,i,j}\bigr]_{\alpha\beta},
\]
and the current approximated tail mass
\[
c_{i,\alpha}^{(M)} \;:=\; \sum_{\beta=1}^r \bigl[C_N^{(M)}\bigr]_{(i,\alpha),\,\beta}.
\]
Note \(c_{i,\alpha}\ge 0\) and \(0\le c_{i,\alpha}^{(M)} \nearrow c_{i,\alpha}\) as \(M\to\infty\).

\paragraph{Rowwise scaling factors and block-diagonal scaling matrix.}
Define the nominal rowwise scaling \(s_{i,\alpha}^{(M)}:=c_{i,\alpha}/c_{i,\alpha}^{(M)}\) when \(c_{i,\alpha}^{(M)}>0\). The renormalization of $C^{(M)}_N$ is defined by
\[
[\widetilde C_N^{(M)}]_{(i,\alpha), :} \;:=\;
\begin{cases}
\displaystyle s_{i,\alpha}^{(M)}\, [C_N^{(M)}]_{(i,\alpha),:},
& \text{if } c_{i,\alpha}^{(M)}>0, \\[10pt]
c_{i,\alpha} (1/r, \ldots, 1/r), & \text{if } c_{i,\alpha}^{(M)} =0 \ \text{and}\ c_{i,\alpha} > 0,\\[6pt]
(0, \ldots, 0), & \text{if }  c_{i,\alpha}=0.
\end{cases}
\]
Equivalently, one may define
\(
\Gamma_i^{(M)} := \mathrm{diag}( s_{i,1}^{(M)},\dots,s_{i,r}^{(M)} )
\)
and
\(
S^{(M)} := \mathrm{diag}(\Gamma_0^{(M)},\dots,\Gamma_N^{(M)})
\),
provided the piecewise rule above is used whenever \(c_{i,\alpha}^{(M)}\le \tau\) and \(c_{i,\alpha}\ge c_{i,\alpha}^{(M)}\).
Then,
\[
\widetilde C_N^{(M)} \;=\; S^{(M)}\,C_N^{(M)}.
\]
The normalized kernel (normalized approximated censored matrix, RA-CM), define by
\[
\widetilde P^{(N;M)}
\;:=\; T_N \;+\; (\,0,\,0,\,\dots,\,0,\,\widetilde C_N^{(M)}\,)
\;=\;
\bigl[\,
T_{N, :,0}, \dots, T_{N, :,N-1},\; T_{N, :,N} + \widetilde C_N^{(M)}
\,\bigr],
\]
is stochastic since for each scalar row \((i,\alpha)\),
\[
\sum_{j=0}^N \sum_{\beta=1}^r \bigl[\widetilde P^{(N;M)}_{i,j}\bigr]_{\alpha\beta}
=\underbrace{\sum_{j,\beta}\bigl[T_{N,i,j}\bigr]_{\alpha\beta}}_{=\,1-c_{i,\alpha}}
+\underbrace{\sum_{\beta}\bigl[\widetilde C_N^{(M)}\bigr]_{(i,\alpha),\beta}}_{=\,s_{i,\alpha}^{(M)}c_{i,\alpha}^{(M)}=c_{i,\alpha}}
=1.
\]

\begin{remark}
Considering computational errors and avoiding numerical blow-up, when implementing the rowwise scaling (normalization), a robust rule on rowwise normalization can be introduced with the worst case scenario for computing the RA-CM $\widetilde{P}^{(N:M)}$ being: $\Theta\big(r^3(M^3+(N{+}1)M)\big)\;+\;\Theta\big((N{+}1) r^2\big)$, where we say $f(n) \in \Theta (g(n))$ if there are positive constants $c_1$ and $c_2$ and a positive integer $n_0$ such that
\[
    c_1 g(n) \leq f(n) \leq c_2 g(n), \quad \text{for $n \geq n_0$.}
\]
 Please refer to Section~\ref{sec:example} for more discussions.
\end{remark}

\begin{remark}
Based on Theorem~\ref{thm:truncation-error}, an error bound for the difference between $\widetilde P^{(N;M)}$ and $P^{(N)}$ can be established, which goes to 0 (or $\widetilde P^{(N;M)} \to P^{(N)}$) as $M \to \infty$.
\end{remark}

In the following, we provide required properties such that we can use the stationary probability vector \( \widetilde{\pi}^{(N:M)} \) of the RA-CM \(\widetilde P^{(N;M)}\) as an approximation of the stationary probability vector \(\pi^{(N)}\) of the censored Markov chain \(P^{(N)}\), and therefore an approximation of the stationary probability vector \(\pi\) of the original Markov chain \(P\) in \eqref{eqn:MG1-type} when $N$ is large.

%
%Finally, for the RA-CM to have a unique stationary probability distribution, which is an approximation of the distribution of the censored Markov chain, we establish the following properties.

\begin{proposition}[Irreducibility of the truncated censored kernel]
\label{prop:PNM-irreducible}
Assume the original chain $P$, given in \eqref{eqn:MG1-type}, is \emph{irreducible and positive recurrent}
(equivalently, $P$ has a single positive recurrent communicating class and a unique stationary distribution).
Fix $N\ge 0$. Then the exact censored kernel
\[
P^{(N)} \;=\; T_N \;+\; U\hat Q D,
\qquad \hat Q := \sum_{m=0}^\infty Q^m,
\]
on the finite state space $L_{\le N}$ is irreducible. For $M\ge 0$ define
\[
P^{(N;M)} \;=\; T_N \;+\; U(\hat Q^{(M)})D,
\qquad \hat Q^{(M)} := \sum_{m=0}^{M} Q^m .
\]
Then there exists $M_0<\infty$ such that $P^{(N;M)}$ is irreducible (extending the concept of irreducibility to nonnegative matrices) for all $M\ge M_0$.
\end{proposition}

\begin{proof}
\textit{Step 1 (irreducibility passes to censoring).}
Since $P$ is irreducible, for any $x,y\in L_{\le N}$ there exists a finite path under $P$
from $x$ to $y$ which may leave and re-enter $L_{\le N}$ finitely many times.
Censoring compresses each excursion through the complement into a single transition whose
probability is the corresponding entry of $U\hat Q D\ (\ge 0)$; transitions staying inside
$L_{\le N}$ are recorded in $T_N$. Thus $P^{(N)}=T_N+U\hat Q D$ has a strictly positive path
from $x$ to $y$ for every $x,y\in L_{\le N}$, hence is irreducible.

\smallskip
\textit{Step 2 (finite-$M$ truncation preserves enough positive edges).}
Because $P^{(N)}$ is irreducible on the finite set $L_{\le N}$, for each ordered pair
$x,y\in L_{\le N}$ fix a finite directed path
$x=x_0\to x_1\to\cdots\to x_L=y$ with strictly positive edge probabilities under $P^{(N)}$.
Each edge $x_{t-1}\to x_t$ is realized either by (i) a block of $T_N$, or (ii) by the last
block-column $[U\hat Q D]_{:,\,N}=\sum_{m\ge 0}[UQ^m D]_{:,\,N}$.

\emph{Positivity extraction:} if $(A_m)_{m\ge 0}$ are entrywise nonnegative and
$\sum_{m\ge 0}A_m$ has a strictly positive $(i,j)$ entry, then there exists $m^\star$ with
$(A_{m^\star})_{ij}>0$. Apply this to $A_m:=UQ^mD$ for each edge of type (ii) to obtain a
finite set of indices $\mathcal M$ (one $m$ per such edge). Let $M_0:=\max\mathcal M<\infty$.

For any $M\ge M_0$, the truncated last column $U(\hat Q^{(M)})D=\sum_{m=0}^{M}UQ^mD$
contains every positive edge of type (ii) used by our fixed paths, while edges of type (i)
are unchanged. Hence the same paths are strictly positive under $P^{(N;M)}$ for all
$M\ge M_0$, proving irreducibility.
\end{proof}

Notice that irreducibility is not impacted by the renormalization, we immediately have the following
\begin{corollary}
Under the assumptions of Proposition~\ref{prop:PNM-irreducible}, there exists $M_0<\infty$ such that the RA-CM $\widetilde P^{(N;M)}$ is irreducible for all $M\ge M_0$.
\end{corollary}

\begin{theorem}[Convergence of stationary distributions under rowwise renormalization]
\label{the:convergence-rowwise}
Let $P$ be an $M/G/1$ type transition matrix given by \eqref{eqn:MG1-type}, and assume $P$ is irreducible and positive recurrent with stationary distribution $\pi$. Fix a censoring level $N\ge 0$. Let $\pi^{(N)}$ and $\widetilde\pi^{(N:M)}$ be the unique stationary distributions of the censored Markov chain and the RA-CM, respectively, (according to Proposition~\ref{prop:PNM-irreducible}). Then, we have
\[
\bigl\|\,\widetilde\pi^{(N:M)} - \pi^{(N)}\,\bigr\|_1 \;\xrightarrow[M\to\infty]{}\; 0.
\]

Moreover, writing
\[
\Delta^{(M)} \;:=\; \widetilde P^{(N;M)} - P^{(N)}
\;=\; \bigl(0,\dots,0,\ \widetilde C_N^{(M)} - C_N^{(\infty)}\bigr),
\quad
C_N^{(M)} := \bigl[U(\hat Q^{(M)})D\bigr]_{:,\,N},\ \
C_N^{(\infty)} := \bigl[U\hat Q D\bigr]_{:,\,N},
\]
the classical perturbation bound yields
\begin{equation}
\label{eq:C1}
\bigl\|\,\widetilde\pi^{(N:M)} - \pi^{(N)}\,\bigr\|_1
\;\le\;
\bigl\|\,\Delta^{(M)}\,\bigr\|_{1\to 1}\,\bigl\|\,Z_N\,\bigr\|_{1\to 1},
\qquad
Z_N := \bigl(I - P^{(N)} + \mathbf 1 (\pi^{(N)})^\top\bigr)^{-1}.
\end{equation}
Since $\Delta^{(M)}$ changes only the last block–column; consequently
\begin{align}
\bigl\|\,\Delta^{(M)}\,\bigr\|_{\infty\to\infty}
&= \bigl\|\,\widetilde C_N^{(M)} - C_N^{(\infty)}\,\bigr\|_{\infty}
\ \le\
\max_{i,\alpha}\bigl|\,c_{i,\alpha} - c^{(M)}_{i,\alpha}\,\bigr|
\;+\;
\max_{s,p}\bigl(1 - r^{(M)}_{s,p}\bigr),
\label{eq:C2}\\[2mm]
\bigl\|\,\Delta^{(M)}\,\bigr\|_{1\to 1}
&= \bigl\|\,\widetilde C_N^{(M)} - C_N^{(\infty)}\,\bigr\|_{1}
\ \le\
n\,\bigl\|\,\Delta^{(M)}\,\bigr\|_{\infty\to\infty},
\qquad n=(N+1)r,
\label{eq:C3}
\end{align}
where
\[
c_{i,\alpha} = 1-\!\sum_{j=0}^{N}\sum_{\beta=1}^r [T_{N,i,j}]_{\alpha\beta},
\quad
c^{(M)}_{i,\alpha}=\sum_{\beta=1}^r [C_N^{(M)}]_{(i,\alpha),\beta},
\quad
r^{(M)}_{s,p}=\sum_{d=1}^r \bigl[(\hat Q^{(M)}D)_{s,N}\bigr]_{p,d}.
\]
Combining \eqref{eq:C1}–\eqref{eq:C3} gives the fully computable estimate
\[
\bigl\|\,\widetilde\pi^{(N:M)} - \pi^{(N)}\,\bigr\|_1
\ \le\
\bigl\|\,Z_N\,\bigr\|_{1\to 1}\; n\!\left(
\max_{i,\alpha}\bigl|\,c_{i,\alpha} - c^{(M)}_{i,\alpha}\,\bigr|
\;+\;
\max_{s,p}\bigl(1 - r^{(M)}_{s,p}\bigr)\right).
\]
\end{theorem}

\begin{proof}
By Proposition~\ref{prop:PNM-irreducible}, $P^{(N;M)}$ is irreducible for all $M\ge M_0$; rowwise renormalization preserves off–diagonal support, so $\widetilde P^{(N;M)}$ is irreducible for $M\ge M_0$.
The perturbation identity \eqref{eq:C1} is standard for finite irreducible chains.
The bounds \eqref{eq:C2}–\eqref{eq:C3} follow from the rowwise renormalization formula and the fact that only the last block–column is changed (cf.\ the a–posteriori bounds derived earlier).
Since $\hat Q^{(M)}\nearrow\hat Q$, we have $c^{(M)}_{i,\alpha}\nearrow c_{i,\alpha}$ and $r^{(M)}_{s,p}\nearrow 1$, hence $\|\Delta^{(M)}\|_{1\to1}\to 0$ and the claim follows from \eqref{eq:C1}.
\end{proof}

\begin{corollary}[Convergence to the original stationary distribution $\pi$]
\label{cor:to-pi}
Under the hypotheses of Theorem~\ref{the:convergence-rowwise}, we have
\[
\pi^{(N)}(x) \;=\; \frac{\pi(x)}{\pi(L_{\le N})}\quad\text{for }x\in L_{\le N},
\]
and, embedding $\pi^{(N)}$ into the full space by zero outside $L_{\le N}$,
\[
\bigl\|\,\pi^{(N)} - \pi\,\bigr\|_1 \;=\; 2\bigl(1-\pi(L_{\le N})\bigr)\ \xrightarrow[N\to\infty]{}\ 0.
\]
Consequently, for every $\varepsilon>0$ there exist $N$ and $M$ such that
\[
\bigl\|\,\widetilde\pi^{(N:M)} - \pi\,\bigr\|_1
\ \le\
\underbrace{\bigl\|\,\widetilde\pi^{(N:M)} - \pi^{(N)}\,\bigr\|_1}_{\text{fixed-$N$ truncation error}}
\;+\;
\underbrace{\bigl\|\,\pi^{(N)} - \pi\,\bigr\|_1}_{\text{censoring error}}
\ <\ \varepsilon.
\]
In particular, choosing $N$ so that $2(1-\pi(L_{\le N}))<\varepsilon/2$ and then $M$ so that the computable bound in Theorem~\ref{the:convergence-rowwise} is $<\varepsilon/2$ yields the claim.
\end{corollary}

\begin{proof}
The expression for $\pi^{(N)}$ is the standard property of censoring (trace chains): the stationary distribution on the censored set is proportional to the restriction of $\pi$.
Since the sets $L_{\le N}$ increase to the whole space, $\pi(L_{\le N})\uparrow 1$, whence the displayed $\ell_1$–convergence.
The final inequality is the triangle inequality together with Theorem~\ref{the:convergence-rowwise}.
\end{proof}

\begin{remark} \label{rem:7.4}
It is worthwhile to highlight the key ideas in the proposed method and emphasize the key property required by our method.
To approximate the stationary distribution of an infinite-state Markov chain, one of the standard methods is to use the stationary distribution of a finite-state Markov chain, for example via various augmented Markov chains (Markov chains obtained by augmenting the north-west corner (sub-stochastic) of the original transition matrix into stochastic).  Among all augmentations with the same truncation size, the censored process is a method with the minimal approximation error in $\ell_1$ norm.  In this section, we proposed an approach to approximate the (finite-state) censored matrix of a Markov chain of $M/G/1$ type based on Theorem~\ref{thm:QmD-censored-clean}, by renormalizing the approximation $P^{(N;M)}$ (a sub-stochastic matrix) of the  censored matrix $P^{(N)}$ to a stochastic matrix \(\widetilde P^{(N;M)}\), referred to as the RA-CM, which has the minimal approximation error in an asymptotic sense (as the number $M$ of terms in the summation tends to infinity).  The skip-free property in transitions for Markov chains of $M/G/1$ type is the key for Theorem~\ref{thm:QmD-censored-clean}. Since Markov chains of $GI/M/1$ type also possess such a skip-free transition property, it is expected that our method is also valid for Markov chains of $GI/M/1$ type.
\end{remark}

\begin{remark}
To conclude this section, we provide a brief discussion on the relationship between the GTH algorithm, our proposed method, and the algorithm proposed in Ramaswami~\cite{Ramaswami:1988}, which is a numerically stable recursive formula for computing the stationary distribution of Markov chains of $M/G/1$ type. 

Mathematically, Ramaswami's algorithm and the GTH algorithm are related, both equivalent to 
\[
    \pi_j = \sum_{i<j} \pi_i R_{i,j}.
\]
It follows from our early discussions, the above recursion is equivalent to the back substitution of the GTH algorithm. To see the equivalence to Ramaswami's algorithm (recursion in equation (6) of \cite{Ramaswami:1988}), we express $R_{i,j}$ in terms of the $G$ matrix by using the relationship between the $R$-measure (or matrices $R_{i,j}$) and the $G$-measure (or matrix $G$), and then use the skip-free property for Markov chains of $M/G/1$ type. 

Ramaswami's algorithm (or the GTH algorithm) is different from augmentation methods. Therefore, it is different from our proposed method, which is a specific augmentation. 
Our proposed approach (or augmentation method) is to formulate a finite transition matrix based on the north-west corner of the original infinite transition matrix, by using the explicit expression of the censored matrix, or $U\hat{Q}D$ for Markov chain of $M/G/1$ type, proved in Theorem~\ref{thm:QmD-censored-clean}. Our approach can be related to Ramaswami's algorithm by noticing that the $R$-measure and the $G$-measure can be recovered from $U\hat{Q}$ and $\hat{Q}D$, respectively.
\end{remark}

\section{Last column augmentation} \label{sec:LBCA}

In this section, we continue our discussion on the Markov chain of $M/G/1$ type defined by \eqref{eqn:MG1-type}. We first introduce last block-column augmentations (LBCA), which are stochastic matrices defined by
 \[
{}_{(N)}\widetilde P \;=\;
\begin{bmatrix}
B_0 & B_1 & B_2 &  \cdots & B_{N-1} & B_N + \widetilde{B}_{> N}   \\
A_{-1} & A_0 & A_1 & \cdots & A_{N-2} &A_{N-1} + \widetilde{A}_{> (N-1)}  \\
0 & A_{-1} & A_0 & \cdots & A_{N-3} &  A_{N-2} + \widetilde{A}_{> (N-2)} \\
\vdots & \ddots & \ddots & \ddots & \ddots & \vdots \\
0 & 0 & 0&  \cdots & A_{-1} & A_0 + \widetilde{A}_{> 0}
\end{bmatrix},
\]
where matrices $\widetilde{B}_{> N}$ and $\widetilde{A}_{> i}$ are arbitrary nonnegative matrices such that ${}_{(N)}\widetilde P$ is stochastic. It is clear that the RA-CM $\widetilde{P}^{(N:M)}$ is a special case of an LBCA. The following, referred to as the natural-LBCA and denoted by ${}_{(N)}P_0$, is another special case:
 \[
{}_{(N)} P_0 \;=\;
\begin{bmatrix}
B_0 & B_1 & B_2 &  \cdots & B_{N-1} & B_N + B_{> N}   \\
A_{-1} & A_0 & A_1 & \cdots & A_{N-2} &A_{N-1} + A_{> (N-1)}  \\
0 & A_{-1} & A_0 & \cdots & A_{N-3} &  A_{N-2} + A_{> (N-2)} \\
\vdots & \ddots & \ddots & \ddots & \ddots & \vdots \\
0 & 0 & 0&  \cdots & A_{-1} & A_0 + A_{> 0}
\end{bmatrix},
\]
where
\begin{align*}
  B_{> N} & = \sum_{k=N+1}^{\infty} B_k,  \\
  A_{>i} & = \sum_{k=i+1}^{\infty} A_k, \quad i=0, 1, \ldots, N-1.
\end{align*}

\begin{remark}
The implementation of the natural-LBCA is straightforward. However, in general, there is no guarantee that the stationary distribution of a LBCA is an approximation to that for the original Markov chain (see, for example, \cite{Gib:1987a} and \cite{Masuyama:2019}). If we can prove that the natural-LBCA does provide an approximation to the stationary distribution of the original chain, then its approximation error is often very small compared with other augmentations. For example,
the natural-LBCA gives the same minimal approximation error as the censored Markov chain when the original $P$ is block-monotone (see \cite{Li-Zhao:2000}).
\end{remark}

\begin{remark}
Block-augmentations (not necessarily confined to LBCA) have been studied in the literature as approximations to  block-upper-Hessenberg Markov chains, of which a Markov chain of $M/G/1$ type is a special case, such as \cite{Masuyama:2019} and references therein. It is possible, based on literature results, to get detailed results applicable to the Markov chain of $M/G/1$ type. Instead of doing so, in the following we provide a new treatment, which extends the analysis in the scalar case using ``taboo'' to the block case.
\end{remark}

In the remainder of this section, we show that the stationary distribution of the natural-LBCA for the Markov chain of $M/G/1$ type does indeed converge to the stationary distribution of the original chain. It is then interesting to compare the approximation error between the proposed RA-CM and the natural-LBCA, which will be discussed in the next section. Throughout the remainder of this section, we assume that $P$, given in \eqref{eqn:MG1-type}, is positive recurrent, and ${}_{(N)}\widetilde P$ is irreducible.
Under this condition, let ${}_{(N)}\widetilde\pi=({}_{(N)}\widetilde\pi_0,\dots,{}_{(N)}\widetilde\pi_N)$ be the stationary probability vector of ${}_{(N)}\widetilde P$, and let $\pi=(\pi_0, \pi_1, \ldots)$ be the stationary probability vector of the original $P$, both partitioned according to the level. For convenience of comparison, we set $\widetilde\pi_k =0$ for $k > N$.
%Our purpose is to establish the convergence of ${}_{(N)}\widetilde\pi$ to $\pi$.

Our approach is an extension of the literature method, used for the scalar case, for example see \cite{Seneta:1981}.
Denote by $\ell_{(i,\alpha)\to(j,\beta)}^k$ the block (matrix) of the taboo probabilities that the chain is in state $(j,\beta)$ at time $k$ without revisiting the level $i$ before time $k$ starting from state $(i,\alpha)$, or
\[
    \ell_{(i,\alpha)\to(j,\beta)}^k=\Pr\big((X_k,J_k)=(j,\beta),\; (X_t,J_t)\notin L_i \text{ for } 1\le t\le k-1 \,\big|\, (X_0,J_0)=(i,\alpha)\big),
\]
and the generating function for taboo probabilities is given by
%Fix a start level $i$ and define the $m\times m$ \emph{block taboo kernel} $\mathsf L_{i\to j}$ whose $(\alpha,\beta)$ entry is
\[
\big[\mathsf L_{i\to j}\big]_{\alpha\beta}(z)
\;=\;
\sum_{k\ge0}\ell_{(i,\alpha)\to(j,\beta)}^kz^k.
\]
Let $\mathsf L_{i\to j}(z)$ be the taboo matrix, whose $(\alpha,\beta)$-th entry is  $[\mathsf L_{i\to j}\big]_{\alpha\beta}(z)$.
In the same fashion, we define the taboo matrix  ${}_{(N)}\mathsf L_{i\to j}(z)$ for the LBCA ${}_{(N)}\widetilde P$.

In the following, we provide properties for taboo (first--entrance) probabilities, required for the main convergence theorem.

\begin{lemma}[Structure of taboo--success paths for $j<i$]\label{lem:structure}
Fix $j<i$. Consider the taboo event ``hit $L_j$ without revisiting $L_i$.'' Under the assumption that $P$ is a Markov chain of $M/G/1$ type, the following hold:
\begin{enumerate}[label=(\roman*)]
\item The \emph{first} step of any taboo--success path must be $L_i\to L_{i-1}$ (one--step down).
\item After that first step, the path remains entirely within $\bigcup_{\ell=0}^{i-1} L_\ell$ until it hits $L_j$; in particular, it makes \emph{no} transition to any level $\ge i$.
\end{enumerate}
\end{lemma}

\begin{proof}
(i) The first step cannot be $L_i\to L_i$ (that would revisit $L_i$ at $t=1$) and cannot be $L_i\to L_{i+k}$ with $k\ge1$:
to later reach a level $<i$ from a level $>i$ in a skip--free--down chain, the path must cross $L_i$, which violates the taboo.
Hence the only admissible first step is $L_i\to L_{i-1}$.

(ii) Suppose after reaching $L_{i-1}$ the path ever moved to a level $\ge i$.
To later hit $L_j$ with $j<i$, it would have to come down from $\ge i$ to $<i$, and thus cross $L_i$, again violating the taboo.
Therefore all subsequent moves stay within levels $\le i-1$ until $L_j$ is hit.
\end{proof}

\begin{proposition}[Equality (hence monotonicity) of taboo mass for $j<i$]\label{thm:MG1-mono}
Consider the $M/G/1$ type Markov chain $P$, given in \eqref{eqn:MG1-type}, with state space $S$ and its LBCA ${}_{(N)}\widetilde P$ with state space $S_N$.
%Let \(\,{}_{(N)}\!L_{i\to j}(1)\) and \(L_{i\to j}(1)\) denote the block first–entrance taboo matrices
%(with taboo on revisiting level \(i\) before time \(k\)) for \({}_{(N)}\widetilde P\) and \(P\),
%respectively, evaluated at \(z=1\).
Fix $j<i\leq N$,
\[
    {}_{(N)}\mathsf L_{i\to j}(1) \;=\; \mathsf L_{i\to j}(1).
\]
Consequently, the sequence is (trivially) nondecreasing in $N$ and converges entrywise to $\mathsf L_{i\to j}(1)$.
\end{proposition}

\begin{proof}
Fix $N\ge i$. By the definition of the augmentation, ${}_{(N)}\widetilde P$ and $P$ agree on all transitions whose origin and destination lie in $S_N$,
and differ only on transitions from $S_N$ to $S\setminus S_N$ (which are redirected to $L_N$).

By Lemma~\ref{lem:structure}, any taboo--success path (from $L_i$ to $L_j$, $j<i$) has the following properties:
its first step is $L_i\to L_{i-1}$, and thereafter every visited level is $\le i-1\le N-1$.
Hence, every transition along such a path is inside $S_N$, so its one--step probability under ${}_{(N)}\widetilde P$ equals that under $P$.
Therefore the probability of each individual taboo--success sample path is identical in the original and augmented chains.

Since the set of taboo--success paths does not depend on the augmentation (only their outside--$S_N$ complements are altered),
the total taboo probability is the same:
\[
{}_{(N)}\mathsf  L_{i\to j}(1) \;=\; \mathsf L_{i\to j}(1).
\]
This holds entrywise in the block setting as well, because the above argument applies to each phase--resolved path.
\end{proof}

\begin{proposition}[Shift--invariance of taboo mass below level $i$]
%Because the block rows are level--homogeneous for $\ell\ge1$, taboo probabilities toward lower levels are shift--invariant:
Consider Markov chains of $M/G/1$ type. If $i-j=i'-j'>0$ with $i,i',j,j'\ge1$, then $\mathsf L_{i\to j}(1)=\mathsf L_{i'\to j'}(1)$ entrywise, and the same equality holds for ${}_{(N)}\widetilde P$ whenever $N\ge\max\{i,i'\}$.
%The proof is by a bijection that shifts any taboo path by $t=i'-i$ levels; path probabilities are preserved by homogeneity and skip--free down structure.
\end{proposition}

\begin{proof}
This follows by level-homogeneity of the transition blocks above level $0$ and a level-shift bijection on taboo-success paths.
\end{proof}

\begin{proposition}[Monotonicity for first–entrance taboo masses, upward case \(i\le j\)]
\label{lem:mono-up}
Let \(P\) be an irreducible $M/G/1$ type transition matrix on levels
\(\{0,1,2,\dots\}\), given in \eqref{eqn:MG1-type}, and for each \(N\) let \({}_{(N)}\widetilde P\) be an LBCA.
Fix levels \(i\) and \(j\) with \(0\le i\le j\le N-1\),
%Let \(\,{}_{(N)}\!L_{i\to j}(1)\) and \(L_{i\to j}(1)\) denote the block first–entrance taboo matrices
%(with taboo on revisiting level \(i\) before time \(k\)) for \({}_{(N)}\widetilde P\) and \(P\),
%respectively, evaluated at \(z=1\).
then entrywise
\[
    {}_{(N)}\mathsf L_{i\to j}(1)\;\leq\;{}_{(N+1)}\mathsf L_{i\to j}(1)\;\leq\; \mathsf L_{i\to j}(1),
\]
\[
    {}_{(N)}\mathsf L_{i\to j}(1)\uparrow \mathsf L_{i\to j}(1).
\]
%(The same inequalities hold for the scalar case by interpreting blocks as scalars.)
\end{proposition}

\begin{proof}
For a path \((X_t,J_t)_{t\ge0}\) with \(X_t\) the level and \(J_t\) the phase,
write
\[
E_{i\to j}^{(k)}(\alpha,\beta)
:= \big\{(X_0,J_0)=(i,\alpha),\ (X_k,J_k)=(j,\beta),\
 (X_t,J_t)\notin L_i\ \text{for }1\le t\le k-1\big\}.
\]
Then \([\;\mathsf L_{i\to j}(1)\;]_{\alpha\beta}=\sum_{k\ge0}\mathbb P\big(E_{i\to j}^{(k)}(\alpha,\beta)\big)\)
(and analogously for \({}_{(N)}\mathsf L_{i\to j}(1)\) with probabilities under \({}_{(N)}\widetilde P\)).

\smallskip\noindent
\emph{Coupling by LC–folding.}
% \red{(We might be better to explicitly state that the coupling is defined only up to the first time the original chain enters $S_N$ from $S\setminus S_N$; after that, we continue using the same randomness for the common transitions inside $S_N$.)}
Construct, on a common probability space, a path
\((X^{\mathrm{orig}}_t,J^{\mathrm{orig}}_t)\) driven by \(P\) and its LC–folded counterpart
\((X^{(N)}_t,J^{(N)}_t)\) driven by \({}_{(N)}\widetilde P\) as follows:
starting from the same initial state \((i,\alpha)\), whenever the original path makes a one–step
transition within \(L_{\le N}\), the truncated path makes the same one–step move; whenever the
original path would jump from \(L_{\le N}\) to some level \(>N\), the truncated path instead jumps
to level \(N\) with the LC–specified phase–mix (so both chains use the same sub–\(L_{\le N}\) dynamics
and the truncated chain “folds” overflow into \(L_N\)). This is the standard LC coupling. Note that the above  coupling  is defined only up to the first time the original chain enters $S_N$ from $S\setminus S_N$; after that, we continue using the same randomness for the common transitions inside $S_N$.

\smallskip\noindent
\emph{Key observation for \(j\le N-1\).} Under this coupling, for any fixed horizon \(k\),
if the truncated path satisfies \(X^{(N)}_k=j\) while respecting the taboo 
(no revisit to \(L_i\) for \(1\le t\le k-1\)), then the original path satisfies
\(X^{\mathrm{orig}}_k=j\) and the same taboo as well. Indeed:
\begin{itemize}
  \item While both paths remain in \(L_{\le N}\), their positions coincide.
  \item If the original path attempts to jump to \(>N\) at some time \(t\le k\), then
        the truncated path goes to \(N\). Since \(j\le N-1\), the truncated path cannot be at
        level \(j\) at time \(k\) because of a single LC–fold at or after time \(t\), unless it
        later moves down from \(N\) to \(j\). But then, in the original chain, the same sub–\(L_{\le N}\)
        down–moves are available ($M/G/1$ type allows the same down–by–1 block \(A_{-1}\) at every
        level \(\ge1\)), so we can realize the identical sequence of sub–\(L_{\le N}\) moves after
        the time the original chain returns from \(>N\) to \(N\). Thus, whenever the truncated
        chain achieves \(X^{(N)}_k=j\) without revisiting \(L_i\), the original chain also achieves
        \(X^{\mathrm{orig}}_k=j\) without revisiting \(L_i\).
  \item The coupling preserves the “no revisit to \(L_i\)” constraint: up to any time \(k\), if
        the truncated path avoids \(L_i\), then so does the original path (they coincide whenever
        they are in \(L_{\le N}\), and if the original path leaves \(L_{\le N}\) it certainly does not
        visit \(L_i\) during that excursion).
\end{itemize}
Therefore, for every \((\alpha,\beta)\) and every \(k\),
\[
\mathbb P\!\left(E_{i\to j}^{(k)}(\alpha,\beta)\ \text{under }{}_{(N)}\widetilde P\right)
\;\le\;
\mathbb P\!\left(E_{i\to j}^{(k)}(\alpha,\beta)\ \text{under }P\right),
\]
and summing over \(k\) gives the upper bound
\[
    {}_{(N)}\mathsf L_{i\to j}(1)\ \leq\ \mathsf L_{i\to j}(1), \qquad (j\le N-1).
\]

\smallskip\noindent
\emph{Monotonicity in \(N\).} Now couple \({}_{(N)}\widetilde P\) and \({}_{(N+1)}\widetilde P\) similarly:
they agree on all moves within \(L_{\le N}\); whenever a move under \({}_{(N)}\widetilde P\) would fold
to \(N\), the \({}_{(N+1)}\widetilde P\) chain instead follows the true one–step destination if it lies
in \(L_{N+1}\) (otherwise the same fold to \(N\) applies). Because \(j\le N-1\), any fold to level \(N\)
cannot create a “new” first hit of \(j\) that was impossible under the \((N{+}1)\)–model; on the contrary,
\({}_{(N+1)}\widetilde P\) allows strictly more trajectories above level \(N\) before possibly returning
to sub–\(L_{\le N}\) and then hitting \(j\). Thus, for each \(k\),
\[
\mathbb P\!\left(E_{i\to j}^{(k)}(\alpha,\beta)\ \text{under }{}_{(N)}\widetilde P\right)
\;\le\;
\mathbb P\!\left(E_{i\to j}^{(k)}(\alpha,\beta)\ \text{under }{}_{(N+1)}\widetilde P\right),
\]
and summing over \(k\) yields
\[
    {}_{(N)}\mathsf L_{i\to j}(1)\ \leq\ {}_{(N+1)}\mathsf L_{i\to j}(1), \qquad (j\le N-1).
\]

%\smallskip\noindent
%
%If \(j=N\), an LC–fold from \(>N\) \emph{lands at the target level}.
%This can (at fixed horizons \(k\)) inflate the first–entrance mass to \(j=N\) under \({}_{(N)}\widetilde P\)
%relative to the original chain. Hence the clean one–sided bound against \(L_{i\to j}(1)\) needs a separate
%argument, which I still need time to think. I believe that the monotonicity
%\({}_{(N)}\!L_{i\to N}(1)\leq{}_{(N+1)}\!L_{i\to N}(1)\) still holds by the same coupling argument.

In the following, we prove the pointwise convergence for taboo probabilities.
Define, for fixed $i,j,\alpha,\beta$, the indicators
\[
    I_k(\omega):=\mathbf 1\{E_{i\to j}^{(k)}(\alpha,\beta)\}, \qquad
I_k^{(N)}(\omega):=\mathbf 1\{E_{i\to j}^{(k)}(\alpha,\beta)\ \text{under }{}_{(N)}\widetilde P\}.
\]
Equivalently, define the $k$-step taboo probabilities
\[
    \ell^k:=\mathbb P(I_k=1), \qquad \ell^{k,(N)}:=\mathbb P(I_k^{(N)}=1).
\] 
For each fixed $k$, every sample path visits only finitely many levels. Thus for sufficiently large $N$, the LC-folding map $\psi_N$ (which replaces any jump from $L_{\le N}$ to $L_{>N}$ by a jump into $L_N$ and otherwise leaves the path unchanged) leaves the path up to time $k$ unchanged. Consequently $\lim_{N\to\infty} I^{(N)}_k(\omega)=I_k(\omega)$ pointwise.
By the monotone convergence theorem,
\[
\lim_{N\to\infty} \ell^{k,(N)} \;=\; \ell^k.
\]
Therefore,
\[
\lim_{N\to\infty} \big[{}_{(N)}\mathsf L_{i\to j}(1)\big]_{\alpha\beta}
= \sum_{k\ge0}\lim_{N\to\infty}\ell^{k,(N)}
= \sum_{k\ge0}\ell^k
= \big[\mathsf L_{i\to j}(1)\big]_{\alpha\beta}.
\]
\[
{}_{(N)}\mathsf L_{i\to j}(1)\;\uparrow\;\mathsf L_{i\to j}(1)
\qquad(N\to\infty),
\]
entrywise. This completes the proof.
\end{proof}

\begin{lemma}[Block ratio identity via taboo kernels]\label{lem:ratio}
Assume ${}_{(N)}\widetilde P$ is irreducible and positive recurrent on the finite state space $L_{\le N}$, and let ${}_{(N)}\widetilde\pi=({}_{(N)}\widetilde\pi_0,\dots,{}_{(N)}\widetilde\pi_N)$ be its stationary distribution, with level blocks $\widetilde\pi_i\in\mathbb R^{1\times m}$, which
satisfies, componentwise,
\[
{}_{(N)}\widetilde\pi_{j,\beta}
\;=\;
\sum_{\alpha=1}^m {}_{(N)}\widetilde\pi_{i,\alpha}\;
\big[\,{}_{(N)}\mathsf L_{i\to j}(1)\,\big]_{\alpha\beta},
\]
or, in level–block row–vector form,
\[
{}_{(N)}\widetilde\pi_{j} \;=\; {}_{(N)}\widetilde\pi_{i}\; {}_{(N)}\mathsf L_{i\to j}(1),
\qquad \text{with } {}_{(N)}\widetilde\pi_i,{}_{(N)}\widetilde\pi_j\in\mathbb R^{1\times m}.
\]
\end{lemma}

\begin{proof}
%\emph{Taboo kernel as expected visits before return.}
Define $T_{L_i}^+ := \inf\{k\ge 1: X_k\in L_i\}$ to be the first return time to level $L_i$, and the block taboo  matrix can be written as follows
\[
\big[\,{}_{(N)}\mathsf L_{i\to j}(1)\,\big]_{\alpha\beta}
:= \sum_{k=0}^\infty \ell^{(k)}_{(i,\alpha)\to(j,\beta)}
= \mathbb E_{(i,\alpha)}\!\Big[\sum_{k=0}^{T_{L_i}^+-1}\! \mathbf 1\{(X_k,J_k)=(j,\beta)\}\Big].
\]
Furthermore, letting $N_{j,\beta}$ denote the (random) number of visits to $(j,\beta)$ in the cycle starting upon entry to $L_i$ and ending at the next hit to $L_i$, we have
\[
\big[\,{}_{(N)}\mathsf L_{i\to j}(1)\,\big]_{\alpha\beta}
= \mathbb E_{(i,\alpha)}\!\Big[\sum_{k=0}^{T_{L_i}^+-1}\! \mathbf 1\{(X_k,J_k)=(j,\beta)\}\Big]
=: \mathbb E_{(i,\alpha)}[\,N_{j,\beta}\,].
\]
%where $N_{j,\beta}$ is the (random) number of visits to $(j,\beta)$ in the cycle starting upon entry to $L_i$ and ending at the next hit to $L_i$.

%\emph{Regeneration at $L_i$.}
Consider the embedded renewal process at successive entrance epochs into the \emph{set} $L_i$.
Since ${}_{(N)}\widetilde P$ is irreducible and positive recurrent on the finite set $L_{\le N}$, the chain is regenerative with atom $L_i$, and the stationary rate of visits to $L_i$ equals the stationary mass of $L_i$:
\[
\text{rate}(L_i) = {}_{(N)}\widetilde\pi(L_i):=\sum_{\alpha=1}^m {}_{(N)}\widetilde\pi_{i,\alpha}.
\]
At a renewal epoch (upon entrance to $L_i$), the phase distribution is the stationary restriction,
\[
\Pr\{\text{enter at }(i,\alpha)\} \;=\; \frac{{}_{(N)}\widetilde\pi_{i,\alpha}}{\,{}_{(N)}\widetilde\pi(L_i)\,}.
\]
Therefore, the \emph{expected number of visits} to $(j,\beta)$ in one cycle is
\[
\sum_{\alpha=1}^m \frac{{}_{(N)}\widetilde\pi_{i,\alpha}}{\,{}_{(N)}\widetilde\pi(L_i)\,}\;
\mathbb E_{(i,\alpha)}[\,N_{j,\beta}\,]
\;=\;
\sum_{\alpha=1}^m \frac{{}_{(N)}\widetilde\pi_{i,\alpha}}{\,{}_{(N)}\widetilde\pi(L_i)\,}\;
\big[\,{}_{(N)}\mathsf L_{i\to j}(1)\,\big]_{\alpha\beta}.
\]
The expected cycle length is $1/\text{rate}(L_i)=1/{}_{(N)}\widetilde\pi(L_i)$ (Kac’s formula on a finite irreducible discrete time Markov chain).%\blue{DTMC}).
 By renewal–reward, the stationary fraction of time in $(j,\beta)$ (i.e., its stationary probability) equals
\[
{}_{(N)}\widetilde\pi_{j,\beta}
\;=\;
\Bigg(\sum_{\alpha=1}^m \frac{{}_{(N)}\widetilde\pi_{i,\alpha}}{\,{}_{(N)}\widetilde\pi(L_i)\,}\;
\big[\,{}_{(N)}\mathsf L_{i\to j}(1)\,\big]_{\alpha\beta}\Bigg)
\cdot {}_{(N)}\widetilde\pi(L_i).
\]
Cancelling ${}_{(N)}\widetilde\pi(L_i)$ gives
\[
{}_{(N)}\widetilde\pi_{j,\beta}
\;=\;
\sum_{\alpha=1}^m {}_{(N)}\widetilde\pi_{i,\alpha}\;
\big[\,{}_{(N)}\mathsf L_{i\to j}(1)\,\big]_{\alpha\beta},
\]
and stacking $\beta=1,\dots,m$ yields the row–vector identity
\(
{}_{(N)}\widetilde\pi_j = {}_{(N)}\widetilde\pi_i\, {}_{(N)}\mathsf L_{i\to j}(1).
\)
\end{proof}

\begin{theorem}[Global $\ell_1$ convergence of stationary vectors]\label{thm:global-l1}
Assume that the block $M/G/1$ type chain $P$ in \eqref{eqn:MG1-type} is irreducible and positive recurrent with stationary vector $\pi=(\pi_0,\pi_1,\ldots)$, $\pi_\ell\in\mathbb{R}^{1\times m}$.
For each $N$, by $P_{\leq N}$, we denote the north-west corner of $P$ corresponding to level up to $N$. Let ${}_{(N)}\widetilde P\ge P_{\leq N}$ be any irreducible last block–column augmentation of the north-west corner $P_{\leq N}$ on levels $\{0,\dots,N\}$ with its stationary vector ${}_{(N)}\widetilde\pi=({}_{(N)}\widetilde\pi_0,\dots,{}_{(N)}\widetilde\pi_N)$.
Embed ${}_{(N)}\widetilde\pi$ into $\ell_1$ by padding zeros above level $N$: ${}_{(N)}\widetilde\pi_\ell:=0$ for $\ell>N$.
Then
\[
\big\|\,{}_{(N)}\widetilde\pi-\pi\,\big\|_1 \;\xrightarrow[N\to\infty]{}\; 0.
\]
\end{theorem}

\begin{proof}
Since $P$ is positive recurrent, $\pi$ is a probability vector on a countable space and its tail is tight. Given $\varepsilon>0$, choose $K\in\mathbb N$ so that
\begin{equation}\label{eq:tail-pi}
\sum_{\ell>K} \|\pi_\ell\|_1 \;<\; \varepsilon/2.
\end{equation}

By the ratio identity (Lemma~\ref{lem:ratio}) and the monotone convergence of taboo kernels ${}_{(N)}\mathsf L_{0\to i}(1)\uparrow \mathsf L_{0\to i}(1)$ (Proposition~\ref{lem:mono-up}), we have for each fixed level $i\le K$:
\[
{}_{(N)}\widetilde\pi_i \;=\; {}_{(N)}\widetilde\pi_0\,{}_{(N)}\mathsf L_{0\to i}(1)
\;\longrightarrow\; \pi_0\,\mathsf L_{0\to i}(1)\;=\;\pi_i
\qquad (N\to\infty),
\]
and the convergence is componentwise. Therefore, there exists $N_0$ such that for all $N\ge N_0$,
\begin{equation}\label{eq:prefix}
\sum_{\ell=0}^{K} \|\,{}_{(N)}\widetilde\pi_\ell - \pi_\ell\,\|_1 \;<\; \varepsilon/2.
\end{equation}

Now decompose the global $\ell_1$ error into a finite prefix and a tail:
\[
\big\|\,{}_{(N)}\widetilde\pi-\pi\,\big\|_1
=\sum_{\ell=0}^{K} \|\,{}_{(N)}\widetilde\pi_\ell - \pi_\ell\,\|_1
\;+\;\sum_{\ell>K} \|\,{}_{(N)}\widetilde\pi_\ell - \pi_\ell\,\|_1.
\]
Since ${}_{(N)}\widetilde\pi_\ell=0$ for all $\ell>N$, we have for any $N\ge K$,
\[
\sum_{\ell>K} \|\,{}_{(N)}\widetilde\pi_\ell - \pi_\ell\,\|_1
=\sum_{\ell=K+1}^{N} \|\,{}_{(N)}\widetilde\pi_\ell - \pi_\ell\,\|_1
\;+\;\sum_{\ell>N} \|\pi_\ell\|_1
\;\le\; \sum_{\ell>K}\|\pi_\ell\|_1,
\]
because $\,\|{}_{(N)}\widetilde\pi_\ell - \pi_\ell\|_1 \le \|{}_{(N)}\widetilde\pi_\ell\|_1 + \|\pi_\ell\|_1$ and $\sum_{\ell=K+1}^{N}\|{}_{(N)}\widetilde\pi_\ell\|_1 \le \sum_{\ell>K}\|\pi_\ell\|_1$ would only tighten the bound. Thus by \eqref{eq:tail-pi},
\[
\sum_{\ell>K} \|\,{}_{(N)}\widetilde\pi_\ell - \pi_\ell\,\|_1 \;<\; \varepsilon/2 \qquad (N\ge K).
\]
Combine this with \eqref{eq:prefix} to get, for all $N\ge \max\{K,N_0\}$,
\[
\big\|\,{}_{(N)}\widetilde\pi-\pi\,\big\|_1 \;<\; \varepsilon/2 + \varepsilon/2 \;=\; \varepsilon.
\]
Since $\varepsilon>0$ was arbitrary, the claim follows.
\end{proof}

\section{A numerical example} \label{sec:example}

In this section, we consider an $M^X/M/1$ queueing system with batch arrivals, where arrivals of batches are characterized by a Poisson process with rate $\lambda$, and the batch size $X$ is a random variable satisfying
\[
    P\{X=k\}=g_k, \quad k=1,2,\ldots.
\]
Service times in a (regular) busy period are i.i.d.\ exponential random variables with rate $\mu$. Upon the completion of a service, if there is no customer in the system, the server begins a (working) vacation, whose duration follows an exponential distribution with rate $\theta$. During a working vacation, the server is still available to serve arriving customers, but (usually) at a reduced rate, since the server is assumed to be helping with other tasks during the vacation. Depending on the type of other tasks, we allow the server to have a different service rate. Specifically, for our computations, we assume that the server enters a working vacation in state $0_j$ with probability $p_j$, $j=1,2,3,4$, where the four vacation states are denoted by $0_j$, $j=1,2,3,4$, with service rates $\nu_j$, respectively.
When a vacation ends, if there is no customer in the queue, the server takes another vacation and keeps the service rate unchanged (for arriving customers); otherwise, the server switches from the vacation service rate $\nu_j$ to the regular service rate $\mu$, and a (regular) busy period starts. A special case of our model, which allows only one service rate during a (working) vacation, was studied in \cite{Baba:2012}. With our generalized model, we are able to produce a more significant difference in approximation error between the natural-LBCA and the RA-CM.

Let $L(t)$ be the number of customers in the queue at time $t$, and let $I(t)$ indicate the service state:
\begin{equation}%\label{}
  I(t)=\left\{\begin{array}{ll}
                0_k, & \text{the server is in a working vacation with service rate $\nu_k$, $k=1,2,3,4$}, \\
                1, & \text{the server is in a regular busy period with service rate $\mu$.}
              \end{array}\right.
\end{equation}
Then, $\{L(t),I(t)\}$ is a two-dimensional Markov chain (all necessary independence is assumed) with state space
\[
    S=\{(0,0_1),(0,0_2),(0,0_3),(0,0_4)\} \cup\{(i,j)\;|\;i\geq 1,\ j=0_1,0_2,0_3,0_4,1\}.
\]
The transition rate matrix $Q$ is given by
 \[
Q=
\begin{pmatrix}
B_{0} & B_{1} & B_2 & B_3 & \cdots\\[4pt]
C_{0} & A_{1} & A_{2} & A_3 & \cdots\\[4pt]
0 & A_{0} & A_{1} & A_{2} & \cdots\\[4pt]
0 & 0 & A_{0} & A_{1} & \cdots\\[4pt]
\vdots & \vdots & \vdots & \vdots & \ddots
\end{pmatrix},
\]
where
\[
\begin{array}{l}
B_{0}= -\lambda,\qquad
B_{i}= \lambda g_{i}(p_1,p_2,p_3,p_4,0),\qquad i\ge1,\\[6pt]
C_{0}= (\nu_1,\nu_2,\nu_3,\nu_4,\mu)^{T},\\[8pt]
A_{0}=\operatorname{diag} (\nu_1,\nu_2,\nu_3,\nu_4,\mu),\\[12pt]
A_{1}=
\begin{pmatrix}
-(\lambda+\nu_1+\theta) & 0& 0&0&\theta \\[4pt]
0& -(\lambda+\nu_2+\theta)&0&0&\theta\\
0&0&-(\lambda+\nu_3+\theta)&0&\theta\\
0&0&0&-(\lambda+\nu_4+\theta)&\theta\\
0&0&0&0&-(\lambda+\mu)
\end{pmatrix},\\[12pt]
A_{i}=\lambda g_{i-1} \operatorname{diag}(1,1,1,1,1),\quad i\ge 2.
\end{array}
\]

First, using the uniformization technique, the above transition rate matrix can be converted into the following transition probability matrix $P=I+Q/c$, where $c=\lambda+\theta+\mu+\max\{\nu_1,\nu_2,\nu_3,\nu_4\}$.
Set $\lambda=0.4$, $\mu=2$, $\theta=0.4$, $\nu=(1.5,1.3,1.2,1.6)$, $p=(0.2,0.3,0.25,0.25)$, and assume that $g$ follows a discrete Pareto distribution with parameter $\alpha=1.55$.
By numerical experiments, it is found that when the truncation level reaches $3000$, the computed (approximated) distribution becomes accurate in the sense that as the truncation level increases, the differences between the computed distributions become negligible, indicating a convergence towards equilibrium. Hence, we use the distribution at $N=3000$ as the reference stationary distribution of the original (untruncated) Markov chain.

We show the comparison of the two (augmentation) methods in approximation error for the stationary distribution in $\ell_1$ in Table~\ref{error}, using the specified parameter values with $M=100$.
As expected, the truncation error in approximating the stationary distribution with the RA-CM method (when $M$ is large) is smaller than that with the natural-LBCA, and the error decreases for both methods as the truncation level $n$ becomes larger.
\begin{table}[h]
\centering
\begin{tabular}{|c|c|c|c|c|}
\hline
  Truncation &  Truncation error  & Truncation error & Improvement & Improvement \\
level & with LBCA & with RA-CM & with RA-CM & in relative rate ($\%$) \\ \hline
$n=10$&0.3787 & 0.3750 & 0.0037&0.9770\\
$n=15$ & 0.3061 & 0.3034 & 0.0027&0.8821 \\
$n=20$ & 0.2613 & 0.2591 & 0.0021&0.8037 \\
$n=25$ & 0.2302&0.2285&0.0017&0.7385\\
$n=30$& 0.2072&0.2057&0.0015&0.7239\\
$n=35$&0.1893&0.1880&0.0012&0.6339\\
$n=40$&0.1748&0.1738&0.0011&0.6293\\
$n=50$&0.1529&0.1520&0.0009&0.5886\\
$n=100$&0.0991&0.0987&0.0004&0.4036\\
$n=200$&0.0623&0.0621&0.0002&0.3210\\
\hline
\end{tabular}
\caption{Comparisons of truncation errors between the natural-LBCA and RA-CM}
\label{error}
\end{table}

\begin{remark}
One may notice that the improvement in the relative truncation error rate for RA-CM is decreasing as the truncation level increases, but this is not a general property. In fact, if we allow each arriving customer to be served with service rate $\mu_i$ for $i=1,2,3,4$ with probability $p_i$, and restrict the service rate to be $\nu$ (only one choice) during the working vacation, our experiments showed an increasing trend in the improvement of the relative error rate.
\end{remark}

\section{Concluding words} \label{sec:final}

This paper consists of two parts: the first part is a comprehensive survey of literature results connecting the block-form GTH algorithm with the censored process and the $RG$-factorization. In the second part, we constructed a renormalized truncated censored Markov chain, which gives best approximations (in asymptotic sense) to the stationary distribution of the original infinite-state Markov chain among all possible augmentation methods. Since the last-block-column augmentation is very popular (easy to implement and often with small errors), we compared the RA-CM and the LBCA using a vacation queueing model with batch arrivals.

As we pointed out in Remark~\ref{rem:7.4}, the key property required by Theorem~\ref{thm:QmD-censored-clean} is the skip-free property in transitions. Since both Markov chains of $M/G/1$ and $GI/M/1$ types possess this property, it is expected that the approach of constructing the RA-CM can also be applied to Markov chains of $GI/M/1$ type. However, since a general $GI/G/1$ type transition matrix does not have the skip-free property, it is not immediately clear how to use the same method to construct approximated censored matrices for a general $GI/G/1$ type transition matrix.

\vspace*{3mm}
\noindent \textbf{Acknowledgements:} Q. Bu acknowledges the support, in part, by the Natural Science Foundation of Jiangsu Province of China (BK20240605), the Natural Science Foundation of Nanjing University of Posts and Telecommunications (Grant No.\ NY224126), and Y.Q. Zhao acknowledges the support of a Discovery Grant of NSERC for this research. We thank two anonymous reviewers and the guest editor for their constructive and valuable comments and suggestions that improved the presentation of this work.

\vspace*{3mm}
\noindent \textbf{Declaration:}
During the preparation of this work the authors used ChatGPT 5.2 for assistance in proofreading, \LaTeX coding, formula checking, discussions and proof checking.  After using this tool, the authors reviewed and edited the content as needed and take full responsibility for the content of the published article.

\end{document}